\documentclass[a4paper,UKenglish,cleveref, autoref, thm-restate]{lipics-v2021}

\pdfoutput=1 
\hideLIPIcs  

\title{Edge Partitions of Complete Geometric Graphs (Part 1)} 

\author{Johannes Obenaus}{Department of Computer Science, Freie Universit{\"a}t Berlin,
Germany}{johannes.obenaus@fu-berlin.de}{https://orcid.org/0000-0002-0179-125X}{Supported by ERC StG 757609.}

\author{Joachim Orthaber}{Institute of Software Technology, Graz University of Technology, Austria}{joachim.orthaber@student.tugraz.at}{https://orcid.org/0000-0002-9982-0070}{}

\authorrunning{J. Obenaus and J. Orthaber} 

\Copyright{Johannes Obenaus and Joachim Orthaber} 



\category{} 

\relatedversion{part 2: \url{https://arxiv.org/abs/2112.08456}} 

\acknowledgements{We are particularly grateful to Patrick Schnider for suggesting the bumpy wheels and his help concerning partitions into plane double stars. Furthermore, we thank the organizers and participants of the $5^{th}$ DACH Workshop on Arrangements and Drawings, that took place in March 2021 (online) and was funded by Deutsche Forschungsgemeinschaft (DFG), the Austrian Science Fund (FWF) and the Swiss National Science Foundation (SNSF).}

\nolinenumbers 

\EventEditors{John Q. Open and Joan R. Access}
\EventNoEds{2}
\EventLongTitle{42nd Conference on Very Important Topics (CVIT 2016)}
\EventShortTitle{CVIT 2016}
\EventAcronym{CVIT}
\EventYear{2016}
\EventDate{December 24--27, 2016}
\EventLocation{Little Whinging, United Kingdom}
\EventLogo{}
\SeriesVolume{42}
\ArticleNo{23}

\usepackage{graphicx}
\graphicspath{{./figures/}}
\usepackage{subcaption}

\usepackage[draft]{todonotes}   

\usepackage{mathtools}
\usepackage{csquotes}

\newcommand{\dist}{\operatorname{dist}}
\newcommand{\spn}{\operatorname{span}}
\newcommand{\obe}{\operatorname{\bar{e}}}
\newcommand{\cl}{\operatorname{cl}}

\usepackage{xspace}
\newcommand{\quasiplanar}{quasi-planar\xspace}

\newcommand{\numPartitionsKThree}{$4^{k-1} + 4^{k-2}$\xspace}

\begin{document}

\maketitle

\begin{abstract}
In this paper, we disprove the long-standing conjecture that any complete geometric graph on $2n$ vertices can be partitioned into $n$ plane spanning trees. Our construction is based on so-called bumpy wheel sets. We fully characterize which bumpy wheels can and in particular which \emph{cannot} be partitioned into plane spanning trees (or even into arbitrary plane \emph{subgraphs}), including a complete description of all possible partitions (into plane spanning trees).

Furthermore, we show a sufficient condition for \emph{generalized wheels} to not admit a partition into plane spanning trees, and give a complete characterization when they admit a partition into plane spanning double stars.
\end{abstract}

\section{Introduction}\label{sec:intro}

A geometric graph $G = G(P,E)$ is a drawing of a graph in the plane where the vertex set is drawn as a point set $P$ in general position (that is, no three points are collinear) and each edge of $E$ is drawn as a straight-line segment between its vertices. A geometric graph $G$ is \emph{plane} if no two of its edges \emph{cross} (that is, share a point in their relative interior).

A \emph{partition} (also called edge partition) of a graph $G$ is a set of edge-disjoint subgraphs of $G$ whose union is $G$. A subgraph of (a connected graph) $G$ is \emph{spanning} if it is connected and its vertex set is the same as the one of $G$.

In 2003, Ferran Hurtado shared the following long-standing open question, which has commonly been conjectured to have a positive answer (see~\cite{op_garden,ferran_2006}):

\begin{question}[\cite{ferran_2006}]\label{ques:hurtado}
Can every complete geometric graph on $2n$ vertices be partitioned into $n$ plane spanning trees?
\end{question}

Note that with $2n>0$ vertices, the complete graph has exactly the right number of edges to admit a partition into $n$ spanning trees, while this is not the case for $2n+1$ vertices. In the following, we consider complete geometric graphs to have $2n$ vertices unless stated otherwise. Further, we denote the complete geometric graph on a point set $P$ as $K(P)$.

\subparagraph*{Related Work.} Several approaches have been made to answer Question~\ref{ques:hurtado}.
When $P$ is in convex position it follows from a result of Bernhart and Kainen~\cite{kainen1979} that $K(P)$ can be partitioned into plane spanning paths, implying a positive answer. Further, Bose et al.~\cite{ferran_2006} gave a complete characterization of all possible partitions into plane spanning trees for convex point sets.
Similarly, when $P=W_{2n}$ is a \emph{regular wheel set} (the vertex set of a regular ($2n-1$)-gon plus its center),
Aichholzer et al.~\cite{aichholzer2017packing} showed how to partition $K(P)$ into plane spanning \emph{double stars} (trees with at most two vertices of degree $\ge 2$), and Trao et al.~\cite{trao2019edge} lately characterized all possible partitions (into arbitrary plane spanning trees).
Further, \cite{aichholzer2017packing} contains a positive answer to \Cref{ques:hurtado} 
for all point sets of (even) cardinality at most $10$, obtained by exhaustive computations.

Relaxing the requirement that the trees must be spanning, Bose et al.~\cite{ferran_2006} showed that if
for a general point set $P$, there exists an arrangement of $k$ lines in which every cell contains at least one point from $P$,
then the complete geometric graph on $P$ admits a partition into $2n-k$ plane trees, $k$ of which are plane double stars.
This result implies that 
\Cref{ques:hurtado} has a positive answer if $P$ contains $n$ pairwise crossing segments, which is the case if and only if $P$ has exactly $n$ \emph{halving lines}~\cite{pach1999halving} (a line through two points of $P$ is called \emph{halving line} if it has exactly $n-1$ points of $P$ on either side and the corresponding edge is called \emph{halving edge}).

For the related \emph{packing} problem where not all edges of the underlying graph must be covered,
Biniaz and Garc{\'{\i}}a \cite{garcia_packing} showed that $\lfloor n/3 \rfloor$ plane spanning trees can be packed in any
complete geometric graph on $n$ vertices, which is currently the best lower bound. Further, in~\cite{aichholzer2017packing} and~\cite{2019_short_packing}, packing plane spanning paths and spanning graphs with short edges, respectively, has been considered.

\subparagraph*{Contribution.} 
In this work, we provide a negative answer to \Cref{ques:hurtado} (refuting the prevalent conjecture). We even provide a negative answer to the following weaker question:

\begin{question}\label{ques:plane_subgraphs}
Can every complete geometric graph on $2n$ vertices be partitioned into $n$ plane \emph{subgraphs}?
\end{question}

Note that the problem of partitioning a geometric graph into plane subgraphs is equivalent to a classic edge coloring problem, where each edge should be assigned a color in such a way that no two edges of the same color cross (of course using as few colors as possible).
This problem received considerable attention from a variety of perspectives (see for example~\cite{DBLP:journals/jct/PawlikKKLMTW14} and references therein) and 
is also the topic of the \href{https://cgshop.ibr.cs.tu-bs.de/competition/cg-shop-2022/\#problem-description}{CG:SHOP challenge 2022} \cite{CG_challenge}.

The point sets in our construction, so-called \emph{bumpy wheel sets}, have been introduced in \cite{patrick_MA, schnider2016packing}. 
For positive odd\footnote{We require $k$ and $\ell$ to be odd for an even number of vertices in total ($k$ has to be odd anyway, since otherwise $W_{k+1}$ would not be in general position).} integers $k$ and $\ell$, the \emph{bumpy wheel} $BW_{k,\ell}$ is derived from the regular wheel $W_{k+1}$ by replacing each of the $k$ hull vertices by a \emph{group} of $\ell$ vertices as follows. 
All vertices (except the center) lie on the convex hull and the vertices within each group are $\varepsilon$-close for some (small enough) $\varepsilon > 0$. 
In particular, the convex hull of any $\frac{k+1}{2}$ consecutive groups does not contain the center vertex (see \Cref{fig:wheels_definition} for an illustration). 
Slightly abusing notation, $BW_{k, \ell}$ refers to the underlying point set as well as the complete geometric graph interchangeably. 
Note that for $\ell = 1$ we obtain a regular wheel set and for $k=1$ a point set in convex position and hence we assume $k, \ell \geq 3$ in the following.

\begin{figure}
\centering
\begin{subfigure}{.25\textwidth}
\centering
\includegraphics[width=\textwidth,page=1]{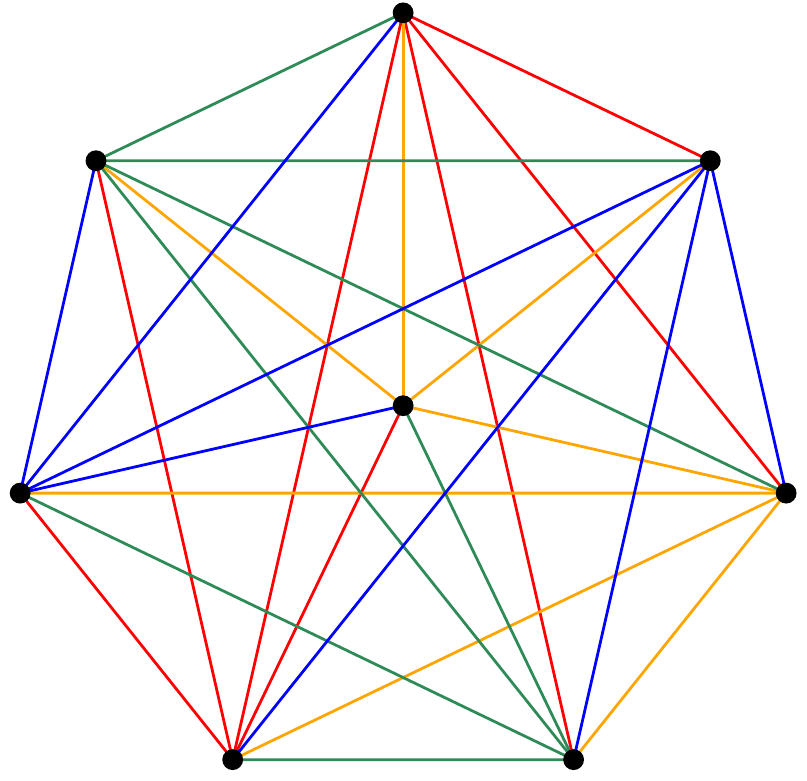}
\label{fig:wheels_definition_a}
\end{subfigure}\hspace{.1\textwidth}%
\begin{subfigure}{.25\textwidth}
\centering
\includegraphics[width=\textwidth,page=2]{wheels_definition}
\label{fig:wheels_definition_b}
\end{subfigure}\hspace{.1\textwidth}%
\begin{subfigure}{.25\textwidth}
\centering
\includegraphics[width=\textwidth,page=3]{wheels_definition}
\label{fig:wheels_definition_c}
\end{subfigure}\hspace{.2\textwidth}%
\caption{A partition of $W_8$ into $n = 4$ plane spanning trees (on the left), the bumpy wheel~$BW_{5,3}$ (in the middle), and the generalized wheel $GW_{[2,3,3,4,5]}$ (on the right).}
\label{fig:wheels_definition}
\end{figure}

Our motivation to study bumpy wheels stemmed from the fact that Schnider~\cite{patrick_MA} showed that $BW_{3,3}$ cannot be partitioned into plane double stars. In contrast, this is always possible for complete geometric graphs on regular wheel sets~\cite{aichholzer2017packing}, as well as complete geometric graphs on point sets admitting $n$ pairwise crossing edges~\cite{ferran_2006} (which also includes convex point sets).

Our first main contribution in this work is to fully characterize for which (odd) parameters $k$ and $\ell$,
the bumpy wheel $BW_{k,\ell}$ can and in particular \emph{cannot} be partitioned into plane spanning trees or plane \emph{subgraphs} (note that also in the subgraph setting we are only interested in partitions into $n$ subgraphs). Surprisingly, allowing arbitrary subgraphs instead of spanning trees does not help much, as it turns out that $BW_{3,5}$ is the only bumpy wheel that can be partitioned into plane subgraphs but not into plane spanning trees.

\begin{theorem}\label{thm:main_spanning_trees}
For odd parameters $k,\ell \geq 3$, the edges of $BW_{k,\ell}$ cannot be partitioned into $n = \frac{k\ell + 1}{2}$ plane spanning trees if and only if $\ell > 3$.
\end{theorem}

\begin{theorem}\label{thm:main_subgraphs}
For odd parameters $k,\ell \geq 3$, the edges of $BW_{k,\ell}$ cannot be partitioned into $n = \frac{k\ell + 1}{2}$ plane \emph{subgraphs} if and only if $\ell > 5$ or $(\ell = 5$ and $k > 3)$.
\end{theorem}

In addition we give a complete description of all non-isomorphic partitions of $BW_{k,3}$ into plane spanning trees (as it has been done for convex point sets~\cite{ferran_2006} and regular wheel sets~\cite{trao2019edge}).

We further consider the more general case of complete geometric graphs on point sets with exactly one point inside the convex hull. To this end, let $GW_N$, for $N = [n_1, \ldots, n_k]$ and integers $n_i \geq 1$, denote the \emph{generalized wheel} with group sizes $n_i$ (in the given circular order). As before, the arrangement of the $k$ groups resembles a regular $k$-gon around the center vertex\footnote{Note that the geometric regularity of generalized wheels is not required (but eases the proofs), as we show in \Cref{app:sec:drop_geometry}.}, the vertices within each group are $\varepsilon$-close, and $k$ is odd. And for our purpose we also require $\sum_i n_i$ to be odd. In this generalized setting, we show a sufficient condition for the non-existence of a partition into plane spanning trees, and give a complete characterization for partitions into plane double stars:

\begin{restatable}{theorem}{generalizedBW}\label{thm:generalized_bw}
Let $GW_N$ be a generalized wheel with $k$ groups and $2n$ vertices. Then $GW_N$ cannot be partitioned into plane spanning trees if each family of $\frac{k-1}{2}$ consecutive groups contains (strictly) less than $n-2$ vertices.
\end{restatable}

\begin{restatable}{theorem}{thmDoubleStars}\label{thm:double_stars}
Let $GW_N$ be a generalized wheel with $k$ groups and $2n$ vertices. Then $GW_N$ cannot be partitioned into plane spanning double stars if and only if there are three families of $\frac{k-1}{2}$ consecutive groups, each of which contains at most $n-2$ vertices, such that each group is in at least one family.
\end{restatable}

We remark that
it is straightforward to model the problem of partitioning into (plane) subgraphs as an integer linear program (ILP), which easily computes solutions for point sets up to roughly 25 points. 
None of the proofs in this paper relies on the computer assisted ILP. But as it served as a great source of inspiration, we describe it in more detail in \Cref{sec:ILP}. 
 
In \Cref{sec:BW}, we prove \Cref{thm:main_spanning_trees,thm:main_subgraphs}, and in \Cref{sec:generalized_bw} we generalize our ideas, proving \Cref{thm:generalized_bw,thm:double_stars}.

\section{Bumpy Wheels}\label{sec:BW}

For a graph in (bumpy) wheel configuration we denote the center vertex by $v_0$ and the remaining vertices by $v_1,\ldots, v_{2n-1}$ in clockwise order. We also enumerate the groups in clockwise order: for $i \in \{1,\ldots,k\}$, $\mathcal{G}_i$ denotes the $i$'th group ($\mathcal{G}_1$ contains $v_1$, $\mathcal{G}_k$ contains $v_{2n-1}$)\footnote{We will consider the index of a group $\mathcal{G}_x$ always modulo $k$, but tacitly mean $((x-1) \mod k) + 1$ (since our indexing starts with 1). The same holds for any other objects, e.g., the vertices on the convex hull.}. An edge having $v_0$ as an endpoint is called \emph{radial edge}, an edge on the convex hull is called \emph{boundary edge} and all other edges are called \emph{diagonal edges}. For a non-radial edge $e$, we define~$e^-$ to be the open halfplane defined by (the supporting line through) $e$ and not containing $v_0$, and similarly $e^+$ to be the open halfplane containing $v_0$.

Additionally, we define a partial order $<_{c}$ on the set of non-radial edges, where $e <_{c} f$ if (the relative interior of) $e$ completely lies in $f^-$ (that is, $f$ is \enquote{closer} to the center vertex $v_0$ than $e$). Two non-radial edges $e,f$ are \emph{incomparable} with respect to $<_c$, if neither $e <_c f$ nor $f <_c e$ holds (we omit \enquote{with respect to $<_c$} if it is clear from the context). In the following, when speaking of an edge $e$ lying in $f^-$ or in $f^+$ for another edge $f$, we always refer to the relative interior of $e$ (that is, an endpoint of $e$ may lie on the line through $f$ --- which actually means to coincide with an endpoint of $f$). A non-radial edge $e$ is \emph{maximal} in some set of edges~$E$, if there is no other edge $e' \in E$ such that $e <_c e'$ (in the following we often consider maximal diagonal edges of plane spanning trees). \emph{Minimal} edges are defined similarly. See \Cref{fig:bw_definitions1} for an illustration. Let us emphasize that we never use $<_c$ for radial edges. 

\begin{figure}
\centering
\includegraphics[scale=0.3,page=1]{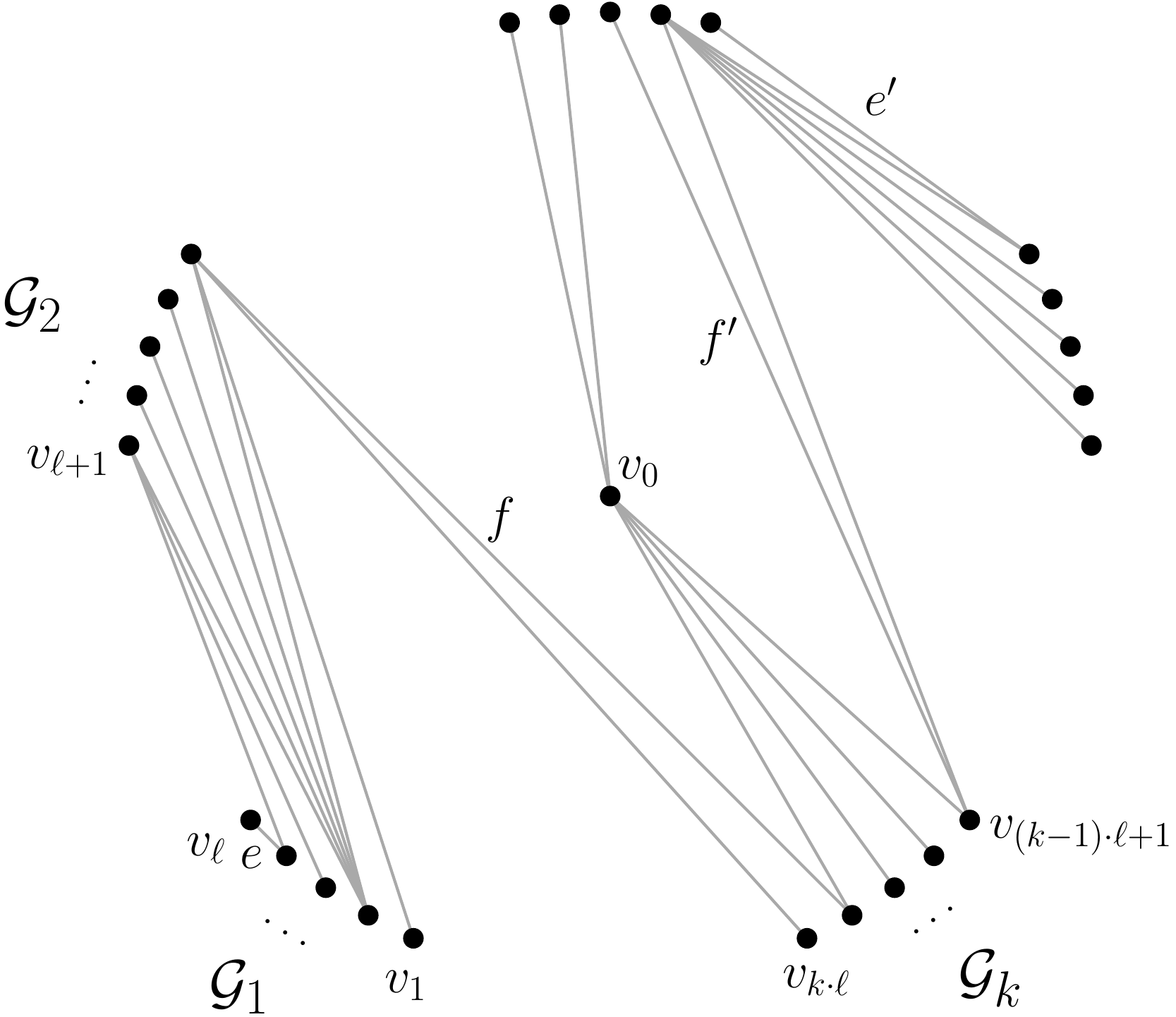}
\caption{Example of a plane spanning tree on the bumpy wheel set $BW_{5,5}$. The diagonal edges $f$ and $f'$ are maximal. The edges $e$ and $e'$ are boundary edges (they are also the only minimal edges).}
\label{fig:bw_definitions1}
\end{figure}

\subsection{Partition into Plane Spanning Trees}\label{sec:tree_disprove}

In this section, we prove \Cref{thm:main_spanning_trees}. We remark that the non-existence direction almost follows from \Cref{thm:main_subgraphs} (not even a partition into plane subgraphs is possible). The only case that is not covered is $BW_{3,5}$, which one can easily verify using computer assistance. However, since the proof of \Cref{thm:main_spanning_trees} is more instructive and intuitive, we decided to present it anyway.
We start with the non-existence direction:

\begin{restatable}{theorem}{disproveConjectureGeneral}
\label{thm:disprove_conjecture_general}
For any odd parameters $k \geq 3$ and $\ell \geq 5$, the edges of $BW_{k,\ell}$ cannot be partitioned into $n = \frac{k\ell + 1}{2}$ plane spanning trees.
\end{restatable}

Towards the proof of \Cref{thm:disprove_conjecture_general}, we will first prove several structural results concerning the number and arrangement of radial and diagonal edges in the spanning trees of a potential partition (some of which have a similar flavor as those in \cite{trao2019edge}). We show that radial edges must lie between maximal diagonal edges and those maximal diagonal edges need to fulfill certain distance constraints. We will show that this cannot be satisfied if $\ell \geq 5$.

The following observation follows immediately from the construction of bumpy wheel sets and the definition of the partial order $<_{c}$.

\begin{observation}\label{obs:maximal_edges_disjoint_vertices}
For two non-radial, non-crossing, incomparable edges $e$, $f$ the vertices in $e^-$ and $f^-$ are disjoint and neither $e^-$ nor $f^-$ contains an endpoint of the other edge.
\end{observation}

Recall that we always refer to \emph{proper} crossings when speaking of crossing edges, that is, $e$ and $f$ in the above observation may share an endpoint. Also note that for any set of edges~$E$, two maximal edges $e,e' \in E$ are always incomparable. 

\begin{restatable}{lemma}{lemPreliminaryOne}\label{lem:preliminary1}
Let $T$ be a plane spanning tree of $BW_{k,\ell}$. Then the following properties hold:
\begin{romanenumerate}
\item\label{lem:prel1:a} for any diagonal edge $e \in E(T)$, $T$ contains at least one boundary edge in~$e^-$,

\item\label{lem:prel1:b} for any pair of incomparable diagonal edges $e,f \in E(T)$, the boundary edges of $T$ in $e^-$ and $f^-$ are distinct, and

\item\label{lem:prel1:c} if $T$ contains exactly one maximal diagonal edge, $T$ contains at least $(\frac{k-1}{2}\ell~+~1)$ consecutive radial edges (in particular, all radial edges of $\frac{k-1}{2}$ consecutive groups).
\end{romanenumerate}\
\end{restatable}

\begin{proof}
For part (i), let $f$ be a minimal edge of $T$ in $e^-$. If $f^-$ does not contain any vertex of the input point set, it is a boundary edge and we are done. Otherwise, since $T$ is connected, at least one vertex in $f^-$ has to be connected to an endpoint of $f$, forming a smaller edge, which is a contradiction.

Part (ii) follows immediately from \Cref{obs:maximal_edges_disjoint_vertices} (distinctness) and part (i) (existence).

Concerning part (iii), let $f$ be the maximal diagonal edge of $T$. Clearly, $f^+$ contains the vertices of at least $\frac{k-1}{2}$ consecutive groups and since $f$ is the only maximal diagonal edge, all vertices in $f^+$ need to be reached by radial edges (plus one to connect to $f$).
\end{proof}

Note that any spanning tree in a partition of $BW_{k,\ell}$ contains a maximal diagonal edge, since the star around~$v_0$ clearly cannot be used in such a partition.

\begin{restatable}{proposition}{propPreliminaryTwo}\label{prop:preliminary2}
Let $T_0, \ldots, T_{n-1}$ be a partition of $BW_{k,\ell}$ into plane spanning trees (if it exists). Then exactly one of those trees, say $T_{0}$, contains a single boundary edge and a single maximal diagonal edge and all other $n-1$ trees contain exactly two boundary edges and exactly two maximal diagonal edges each. 
	
In particular, any diagonal edge $e \in E(T_i)$ contains exactly one boundary edge of $T_i$ in~$e^-$.
\end{restatable}

\begin{proof}
Every $T_i$ contains at least one maximal diagonal edge and hence, by \Cref{lem:preliminary1} (\ref{lem:prel1:a}), also at least one boundary edge.

Since there are $2n-1$ boundary edges in total, at least one tree (w.l.o.g. $T_0$) contains exactly one boundary edge. By \Cref{lem:preliminary1} (\ref{lem:prel1:a}) and (\ref{lem:prel1:b}) it also contains exactly one maximal diagonal edge.

Now, if there was a second spanning tree $T_{1}$ in the partition with exactly one maximal diagonal edge, $T_{0}$ and $T_{1}$ together would use at least $(k-1) \cdot \ell + 2$ radial edges (by \Cref{lem:preliminary1}~(\ref{lem:prel1:c})). This leaves at most $\ell-2$ radial edges for the remaining $n-2$ trees; clearly not enough (since $n = \frac{k\ell+1}{2} > \ell$ for $k,\ell \geq 3$).

Hence, all other $n-1$ spanning trees have to contain at least two maximal diagonal edges and therefore at least two boundary edges. However, since we have $2n-1$ boundary edges in total and only one tree contains a single boundary edge, all other $n-1$ trees have to contain exactly two boundary edges. By \Cref{lem:preliminary1} (\ref{lem:prel1:b}), they also contain at most, and therefore exactly, two maximal diagonal edges.
\end{proof}

From now on, $T_0$ always denotes the spanning tree with exactly one boundary edge (when considering a partition into plane spanning trees). Further, we let all radial edges $\{v_0, v_i\}$ for $i \in \{1,2,\ldots,\frac{k-1}{2}\ell+1\}$ be part of $T_0$ (which we can assume without loss of generality due to symmetry).

For  two non-radial, non-crossing edges $e$, $f$, define the \emph{span} of $e$ and $f$ to be the (closed) area between the two edges, that is,
\[
\spn(e,f) = 
\begin{cases}
\cl(e^+ \cap f^+) & \text{if } e \text{ and } f \text{ are incomparable} \\
\cl(e^+ \cap f^-) & \text{if } e <_{c} f,
\end{cases}
\]
where $\cl(\cdot)$ denotes the closure. The shaded area in \Cref{fig:bw_general_counter_definitions} for instance defines the span of two incomparable edges $e$ and $f$.

Note, however, that we are more interested in the vertices and edges contained in the span, rather than the area itself. If we want to emphasize this, we may use the notation $V(\spn(e,f))$ or $E(\spn(e,f))$. In the following we are mostly interested in the span of maximal diagonal edges of some plane spanning tree.

\begin{restatable}{lemma}{lemBWTwoBoundary}\label{lem:bw_2_boundary}
Let $T_0, \ldots, T_{n-1}$ be a partition of $BW_{k,\ell}$ into plane spanning trees (if it exists) and $e,f$ be the maximal diagonal edges of some $T_i$ ($i\neq 0$). Then, all edges of $T_i$ in the span of $e$ and $f$ are radial (except $e$ and $f$ of course), and all radial edges of $T_i$ are in the span of $e$ and $f$.
\end{restatable}

\begin{proof}
Assume $h$ is a non-radial edge in the span of $e$ and $f$. Then, either $h^-$ contains $e$ or $f$ or an additional third boundary edge (or $h$ is a boundary edge itself), a contradiction in any case. Furthermore, any radial edge not contained in $\spn(e,f)$ must cross either $e$ or $f$, and therefore cannot be part of $T_i$, due to planarity.
\end{proof}

Define the \emph{distance} $\dist(e)$ of a non-radial edge $e$ to be the number of vertices in~$e^-$ plus one (or in other words, the number of boundary edges in $\cl(e^-)$). Clearly, $1\leq \dist(e) \leq \frac{k+1}{2}\ell -1$ holds for any non-radial edge $e$ and $\dist(f) < \dist(e)$ holds for any edge $f \subseteq e^-$. It will be convenient to define, for $i \in \{1,\ldots, \frac{k+1}{2}\ell-1\}$:
\begin{equation}\label{eq:d_i}
d_i = \frac{k+1}{2}\ell - i.
\end{equation}
We define it in this (slightly counter-intuitive) way, $d_1$ being the largest distance, since we mostly deal with edges of large distances and thereby aim to improve the readability.

\begin{restatable}{lemma}{lemContinueDiagonalEdges}\label{lem:continue_diagonal_edges}
Consider a plane spanning tree $T$ of a partition of $BW_{k,\ell}$ and let~$e$ be a diagonal edge in $T$ of distance $d = \dist(e) > 1$. Then $T$ also contains exactly one of the edges of distance $d-1$ in $e^-$.
\end{restatable}

\begin{proof}
Let $f$ be a maximal edge among the edges of $T$ in~$e^-$ and assume $\dist(f) \leq d-2$. Then the span of $e$ and $f$ contains at least one (non-radial) edge $h \notin \{e,f\}$ of $T$ (since either $e$ and $f$ have no endpoint in common and therefore need to be connected via edges in $\spn(e,f)$ or $\spn(e,f)$ contains at least one vertex $v$ which is neither an endpoint of $e$ nor $f$ and needs to be connected to the rest of $T$). Since $f$ is maximal in $e^-$, $h$ and~$f$ must be incomparable. So, by \Cref{lem:preliminary1} (\ref{lem:prel1:b}) this forces at least two distinct boundary edges to be contained in $e^-$ (a contradiction to \Cref{prop:preliminary2}).

Therefore, we get $\dist(f) = d-1$ and since the two edges of distance $d-1$ in $e^-$ cross (or form a cycle with $e$ if $\dist(e) = 2$), $T$ cannot contain both.
\end{proof}

We need a little more preparation towards the proof of \Cref{thm:disprove_conjecture_general}. We call the first and last vertex of each group \emph{outmost vertices} (and the corresponding radial edges \emph{outmost radial edges}). Note that there are exactly $2k$ outmost radial edges in $BW_{k, \ell}$. Every hull vertex/radial edge that is not outmost, is called \emph{inside} vertex/radial edge. 

Furthermore, define two groups $\mathcal{G}_i, \mathcal{G}_j$ to be \emph{opposite} if $|i-j| = \frac{k-1}{2}$ or $|i-j| = \frac{k+1}{2}$. In particular, each group has two opposite groups and two consecutive groups have exactly one opposite group in common (we call that group \emph{the} opposite group of a pair of consecutive groups). 

Let $e,f$ be two maximal (non-crossing) diagonal edges which have an endpoint in a common group. Then the set of vertices of $\spn(e,f)$ in the common group is called \emph{apex}. Note that any apex contains at least one vertex (and this lower bound is attained if the endpoints of $e$ and $f$ coincide).

Moreover, two maximal (non-crossing) diagonal edges $e = \{u,v\}$ and $f = \{u',v'\}$ form a \emph{special wedge} if two endpoints (say $u$ and $u'$) are consecutive outmost vertices of different groups (that is, $u = v_{j\ell}$ and $u' = v_{j\ell+1}$ for some $j$) and $v$ and $v'$ are inside vertices lying in the opposite group of $\mathcal{G}_j$ and $\mathcal{G}_{j+1}$. See \Cref{fig:bw_general_counter_definitions} for an illustration of these terms.

\begin{figure}
\centering
\includegraphics[width=0.35\textwidth,page=1]{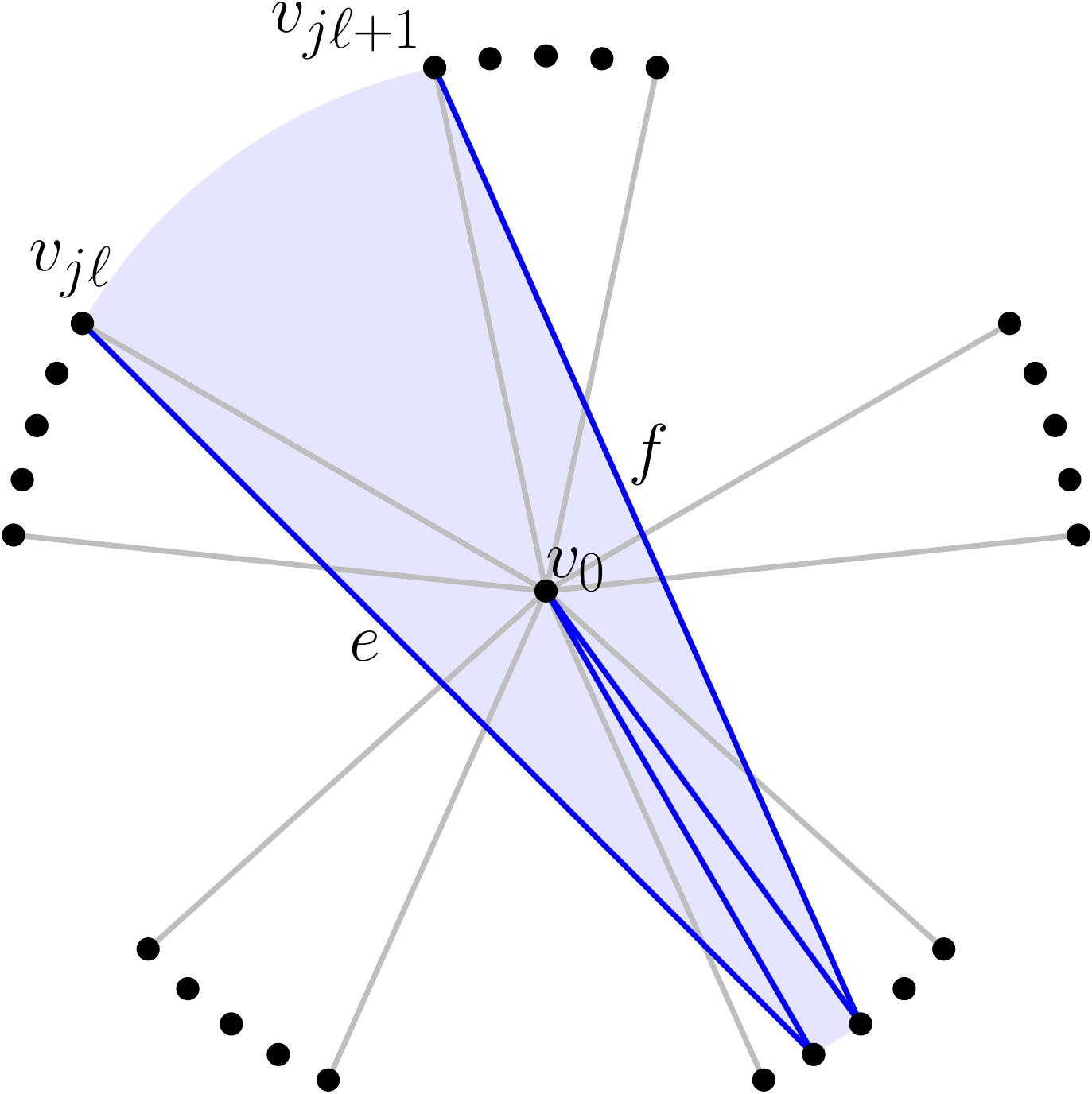}
\caption{All outmost radial edges are depicted in gray. The maximal diagonal edges $e$ and $f$ (connecting opposite groups) form a special wedge. Their span is shaded blue.}
\label{fig:bw_general_counter_definitions}
\end{figure}

\begin{proposition}\label{prop:no_outmost_radial}
Let $T_0, \ldots, T_{n-1}$ be a partition of $BW_{k,\ell}$ into plane spanning trees (if it exists) and let $T_i$ ($i\neq 0$) be a spanning tree that does not use any outmost radial edge. Then the two maximal diagonal edges $e$, $f$ of $T_i$ form a special wedge and $T_i$ has to use all radial edges incident to the apex of this wedge.
\end{proposition}

\begin{proof}
We first argue that all but exactly two radial edges in $\spn(e,f)$ must be part of~$T_i$. The subgraph of $T_i$ induced by $V(\spn(e,f))$ needs to form a tree. Moreover, $\spn(e,f)$ contains $|V(\spn(e,f))| - 1$ radial edges. Since, $T_i$ uses the two diagonal edges $e,f \in E(\spn(e,f))$ and all other edges in the span need to be radial (\Cref{lem:bw_2_boundary}), it has to use exactly all but two radial edges.

Furthermore, since we cannot have two maximal diagonal edges between the same pair of groups, the span of $e$ and $f$ contains at least two outmost vertices, namely in two different groups which contain an endpoint of $e$ and $f$, respectively. On the other hand, $\spn(e,f)$ cannot contain any third outmost vertex and the two outmost vertices contained in $\spn(e,f)$ must be endpoints of $e$ and $f$, since $T_i$ otherwise has to use an outmost radial edge (by \Cref{lem:bw_2_boundary} and above argument). In particular, $e$ and $f$ share a common group and the apex does not contain any outmost vertex (hence, $e$ and $f$ form a special wedge, as depicted in \Cref{fig:bw_general_counter_definitions}). 
 
Moreover, since $T_i$ has to use all but two radial edges in the span, it clearly has to use all radial edges incident to the apex.
\end{proof}

Note that for two spanning trees $T_i$, $T_j$ ($i\neq j$) not using an outmost radial edge, their apexes must be disjoint.

\begin{proposition}\label{prop:one_max_edge_per_distance}
Let $T_0, \ldots, T_{n-1}$ be a partition of $BW_{k,\ell}$ into plane spanning trees (if it exists). Then for each pair $\mathcal{G}, \mathcal{G}'$ of opposite groups and each $j \in \{1,\ldots, \ell\}$ there is a unique diagonal edge (connecting $\mathcal{G}$ and $\mathcal{G}'$) of distance $d_j$ (recall \Cref{eq:d_i}) that is maximal in its tree.
\end{proposition}

\begin{proof}
Observe first that for any $j \in \{1,\ldots, \ell\}$ there are exactly $j$ edges of distance $d_j$ (between $\mathcal{G}$ and $\mathcal{G}'$) and all edges of the same distance (between $\mathcal{G}$ and $\mathcal{G}'$) pairwise cross. Also note, for any two edges $e$, $e'$ (between $\mathcal{G}$ and $\mathcal{G}'$) with $\dist(e) > \dist(e')$, either $e' \subseteq e^-$ holds or they cross. In particular, if they do not cross and belong to the same tree, the shorter is not a maximal edge.

Consider now for some $j \in \{2,\ldots, \ell\}$ the distance $d_j$ and let $c_1, \ldots, c_j$ be the colors\footnote{Instead of always spelling out that an edge belongs to a plane subgraph, we associate edges with colors.} used for all edges of this distance. By \Cref{lem:continue_diagonal_edges}, there are $j-1$ edges of (the larger) distance $d_{j-1}$ using the same color as an edge of distance $d_j$, w.l.o.g. $c_1, \ldots, c_{j-1}$. By the above arguments the corresponding edges of distance $d_j$ cannot be maximal.

On the other hand, the color $c_j$ cannot be used by any edge of larger distance, since again by \Cref{lem:continue_diagonal_edges} this color would have to appear in $d_{j-1}$ as well. Hence, indeed the only edge of distance $d_j$ that is maximal in its tree is the one of color $c_j$.

Lastly, for $j=1$ observe that the single edge of distance $d_1$ is clearly maximal.
\end{proof}

Finally, we are ready to prove \Cref{thm:disprove_conjecture_general}, which we restate here for the ease of readability: 

\disproveConjectureGeneral*

\begin{proof}
Assume to the contrary that there is such a partition $T_0, \ldots, T_{n-1}$. 
There are $2k$ outmost radial edges and $T_{0}$ uses (at least) $k$ of them (see the remark after \Cref{prop:preliminary2}).
Hence, there are at most $k+1$ spanning trees (including $T_0$) containing an outmost radial edge.

Next, let us count how many spanning trees \emph{not} containing an outmost radial edge we can have. Since, by \Cref{prop:no_outmost_radial}, the apex of such a tree cannot use any outmost vertex nor any vertex already incident to a radial edge in $T_0$, there remain $\frac{k+1}{2}(\ell - 2)$ possible vertices (to be used by apexes), namely the inside vertices of the last $\frac{k+1}{2}$ groups $\mathcal{G}_{\frac{k+1}{2}}, \ldots, \mathcal{G}_k$ (which are not fully connected to $v_0$ by radial edges in $T_0$). Also recall that each apex contains at least one vertex.

It is crucial to emphasize that among those last $\frac{k+1}{2}$ groups, group $\mathcal{G}_{\frac{k+1}{2}}$ and group $\mathcal{G}_{k}$ are opposite (the only opposite pair). Therefore, by \Cref{prop:no_outmost_radial}, two spanning trees with an apex in group $\mathcal{G}_{\frac{k+1}{2}}$ and group $\mathcal{G}_{k}$, respectively, must each have a maximal diagonal edge between these two groups. Hence, by \Cref{prop:one_max_edge_per_distance}, we can have at most $(\ell - 2)$ spanning trees with apex in one of these two groups (instead of $2(\ell - 2)$); see \Cref{fig:bw_most_general_counter}.

So, in total there can be at most $\frac{k-1}{2}(\ell - 2)$ spanning trees which do not use an outmost radial edge. Hence, whenever
\[
k + 1 + \frac{k-1}{2}(\ell - 2) < \frac{k\ell + 1}{2}
\]
holds, we cannot find enough spanning trees. Rearranging terms, this inequality is equivalent to $\ell > 3$.
\end{proof}

\begin{figure}
\centering
\includegraphics[scale=0.4,page=1]{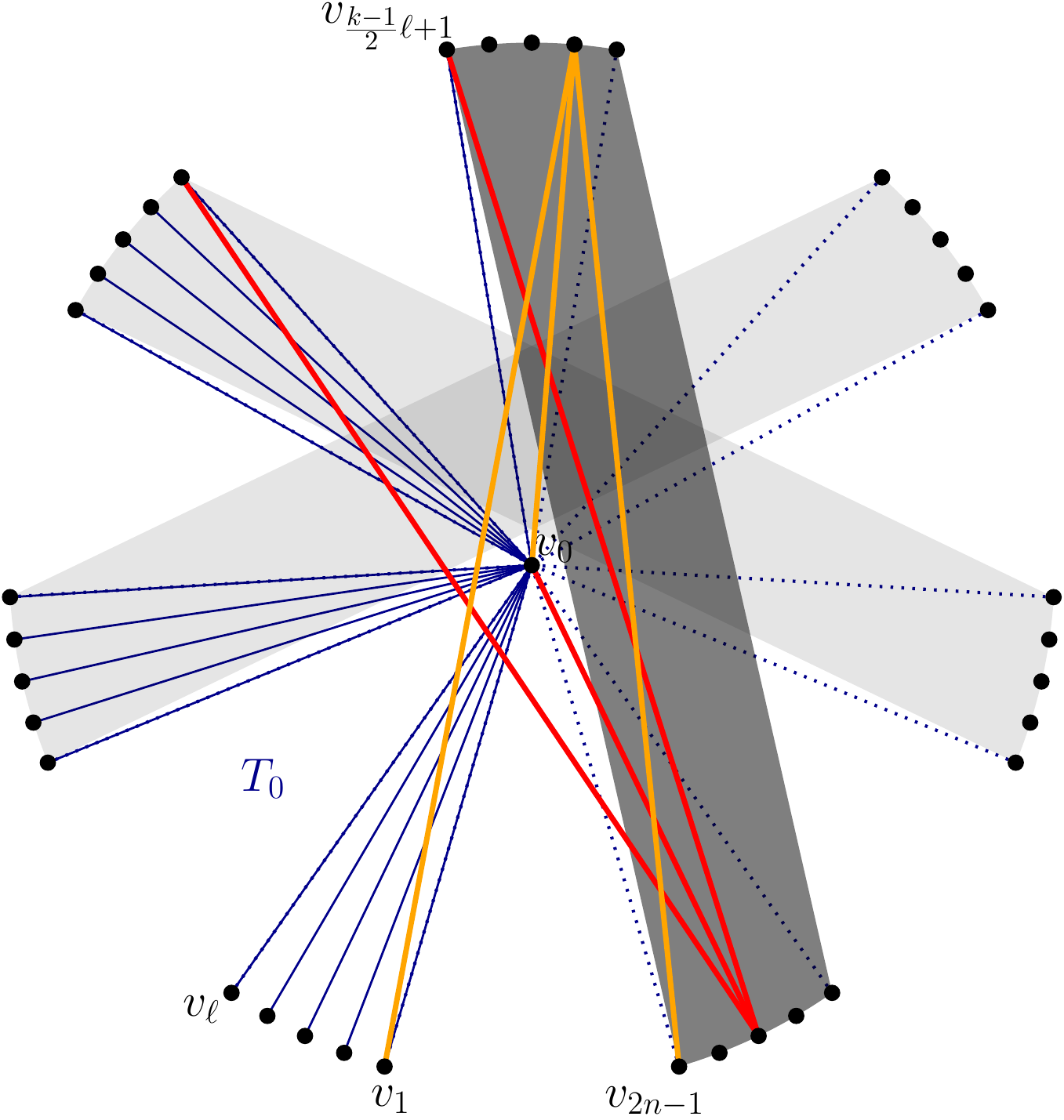}
\caption{In the black stripes (the darker one is the crucial one) the maximal diagonal edges (of those trees without outmost radial edge) need to have distinct distances. That allows $\ell - 2$ many for each stripe. Two spanning trees (red and orange) with apex in group $\mathcal{G}_\frac{k+1}{2}$ and group $\mathcal{G}_k$, respectively, both need to have a maximal diagonal edge in the dark stripe.}
\label{fig:bw_most_general_counter}
\end{figure}

Next, we prove the other direction of \Cref{thm:main_spanning_trees}:

\begin{restatable}{theorem}{thmBWKThreeClassification}\label{thm:bw_k3_classification}
For any odd $k \geq 3$, there are exactly \numPartitionsKThree non-isomorphic partitions of $BW_{k,3}$ into plane spanning trees.
\end{restatable}

Our construction consists of three steps. In the first step, we give explicit constructions of \emph{partial partitions} that cover all radial edges, each (partial) tree in the partition covers exactly its span, and between any pair of opposite groups exactly one diagonal edge of each distance $d_1, d_2, d_3$ is covered. After that we extend these partial partitions in two steps to a full partition (in fact, these extensions work for arbitrary $\ell$, however the partial partitions only exist for $\ell = 3$).

However, we believe it is more instructive to work through these steps in reverse order, that is, we first assume to have a partial partition, and show how to extend it.

\begin{lemma}\label{lem:full_extension}
Let $T_0, \ldots, T_{n-1}$ be a partial partition of the edges of $BW_{k,\ell}$ such that exactly all radial edges and all diagonal edges of distance $d_1, \ldots, d_{\ell}$ are covered, and such that all partial trees are connected, plane, and non-empty. Then, this partial partition can be extended to a partition of $BW_{k,\ell}$ into plane spanning trees. More precisely, there are $2^{\frac{k-1}{2}\ell - 1}$ such possible extensions.
\end{lemma}

\begin{proof}
First of all note that for any $i \geq \ell$, there are exactly $2n-1$ edges of distance $d_i$. So, assume that all edges of distance $d_i$ (for some $i \geq \ell$) are covered, then we want to cover all edges of distance $d_{i+1}$ (the next \emph{smaller} distance). To this end, consider all edges $e_1, e_2, \ldots, e_{2n-1}$ of distance $d_i$ in clockwise circular order (around the center vertex $v_0$). By \Cref{lem:continue_diagonal_edges}, any edge $e_j$ (of distance $d_i$) has an edge of distance $d_{i+1}$ in $e_j^-$, that is, either the one attached to the \enquote{left} endpoint or the one attached to the \enquote{right} endpoint of $e_j$ (viewed from $v_0$). And both choices are possible because no edge of distance $d_{i+1}$ is covered yet. However, note that 
choosing \enquote{left} for the edge $e_j$ also fixes \enquote{left} for the edge $e_{j-1}$ and choosing \enquote{right} for the edge $e_j$ also fixes \enquote{right} for the edge $e_{j+1}$ (see \Cref{fig:bw_k3_proof} for an illustration). Therefore, fixing the choice for one edge, fixes the choice for all edges of distance $d_i$ (because there are equally many edges of distance $d_i$ and $d_{i+1}$).

\begin{figure}
\centering
\begin{subfigure}{.3\textwidth}
\centering
\includegraphics[width=\textwidth,page=1]{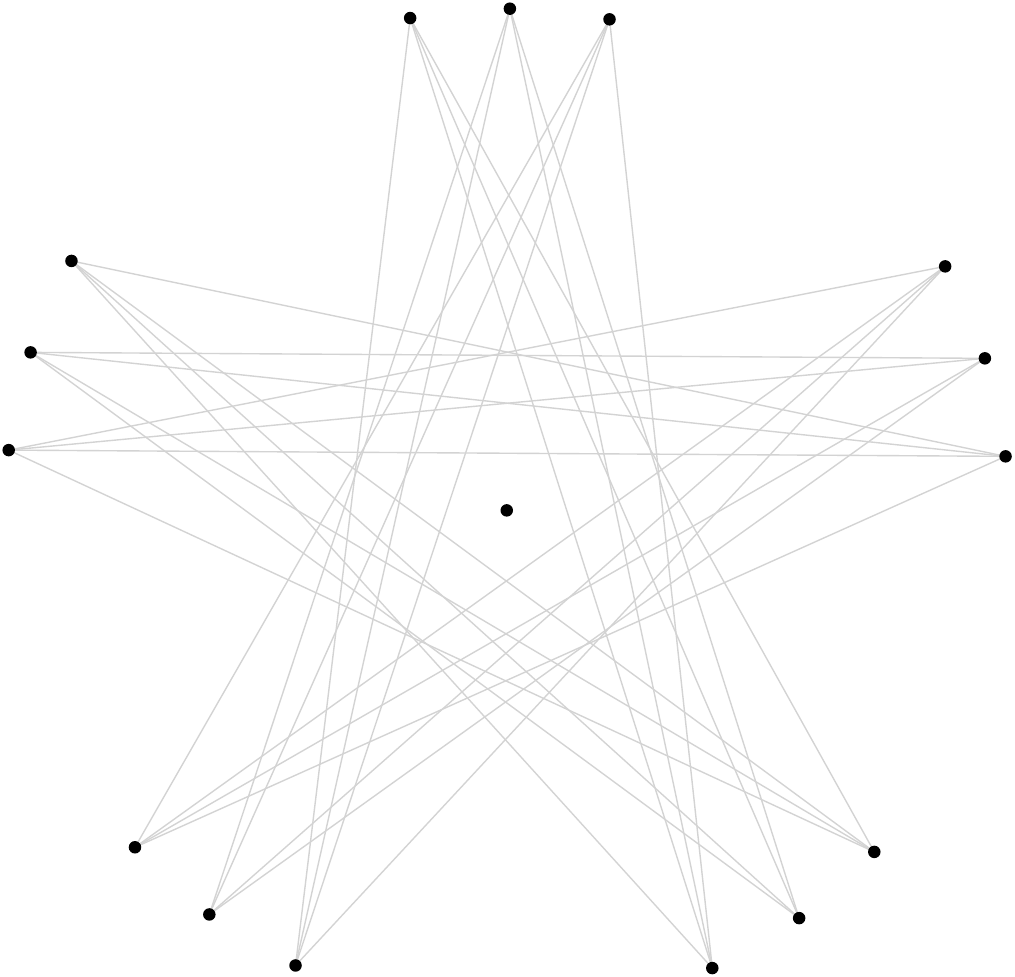}
\caption{\centering}
\label{fig:bw_k3_proof_a}
\end{subfigure}\hspace{2cm}%
\begin{subfigure}{.3\textwidth}
\centering
\includegraphics[width=\textwidth,page=2]{bw_k3_proof}
\caption{\centering}
\label{fig:bw_k3_proof_b}
\end{subfigure}
\caption{(a) All diagonal edges of distance $d_1,\ldots,d_\ell$. (b) All solid edges are of distance $d_i$ and already covered. If the dashed black edge (of distance $d_{i-1}$) is attached to the left endpoint of the blue edge, the red edge has only one choice, namely to also attach the next edge (dashed gray) to the left endpoint.}
\label{fig:bw_k3_proof}
\end{figure}

To be precise, at the start of each iteration, fix an orientation $o~\in~\{\text{left}, \text{right}\}$. Now, for each edge $e_j$ of distance $d_i$ attach to $e_j$ the edge $e_j^o$, which is the edge of distance $d_{i+1}$ attached to the left/right endpoint of $e_j$ (depending on $o$ and viewed from $v_0$). Since there are equally many edges of distance $d_i$ and $d_{i+1}$, every edge of distance $d_{i+1}$ will clearly be added to exactly one tree. And because the partial trees so far are connected, plane, and non-empty, also the extensions will be (still being cycle-free and plane follows from the fact that no edge in $e_j^-$ was covered before, so any edge crossing $e_j^o$ would also cross $e_j$). Continue this process until distance~$d_{i+1}=d_{\frac{k+1}{2}\ell-1}=1$ (having $2$ choices in each step, independently). Then all edges are covered and, since we started with $n$ non-empty partial trees, all $T_i$'s form plane spanning trees.
\end{proof}

\begin{lemma}\label{lem:base_extension}
Let $T_0, \ldots, T_{n-1}$ be a partial partition of the edges of $BW_{k,\ell}$ such that 
\begin{enumerate}[(a)]
\item\label{item:base_extension_a} all radial edges are covered,
\item\label{item:base_extension_b} for each pair of opposite groups, there is exactly one diagonal edge of each distance $d_1, \ldots, d_\ell$ covered, and no further diagonal edges are covered,
\item\label{item:base_extension_c} each of the covered diagonal edges is maximal in its partial tree, and
\item\label{item:base_extension_d} each partial tree is connected, plane, and non-empty.
\end{enumerate}
Then, there is a unique way to extend this partial partition to one that covers all diagonal edges of distance $d_1,\ldots,d_\ell$.
\end{lemma}

\begin{proof}
The proof is very similar to the one of \Cref{lem:full_extension}, with the difference that we need to go through the distances in each pair of opposite groups separately. So, let $\mathcal{G}, \mathcal{G}'$ be a pair of opposite groups and consider some distance $d_i$ with $i \in \{2,\ldots, \ell\}$ (assuming all edges of larger distance between this pair are already covered). Note that by assumption~(b), the one edge of distance $d_1$ is covered (providing a base case), and there is exactly one edge $e$ of distance $d_i$ already covered and $e$ is maximal in its partial tree $T_k$ by assumption~(c) (that is, no edge of distance $d_{i-1}$ between $\mathcal{G}$ and $\mathcal{G}'$ belongs to $T_k$). Now, at least one of the endpoints of $e$ is incident to an edge $f$ of distance $d_{i-1}$ between $\mathcal{G}$ and $\mathcal{G}'$ (it could be both, then name the other $f'$). More precisely, let $f$ be incident to the left endpoint of $e$ (viewed from $v_0$) and colored blue, and $f'$ be incident to the right endpoint of $e$ (viewed from $v_0$) and colored red. The crucial observation is that $e$ now blocks the \enquote{left} extension of the blue tree as well as the \enquote{right} extension of the red tree (viewed from $v_0$). Therefore the red and the blue tree have a unique extension and by using those, they subsequently fix the orientation of the extension for all further trees with an edge of distance $d_{i-1}$ between $\mathcal{G}$ and $\mathcal{G}'$ (to \enquote{left} for all edges to the left of the red edge $f'$ and to \enquote{right} for all edges to the right of the blue edge $f$ --- see \Cref{fig:bw_k3_continuation} for an illustration). Finally, since there is exactly one more edge of distance $d_i$ than of distance $d_{i-1}$ between $\mathcal{G}$ and $\mathcal{G}'$, this uniquely determines the color for all edges of distance $d_i$. And similar to the proof of \Cref{lem:full_extension}, all extended partial trees are still connected, plane, and non-empty.
\end{proof}

\begin{figure}
\centering
\begin{subfigure}{.3\textwidth}
\centering
\includegraphics[width=\textwidth,page=1]{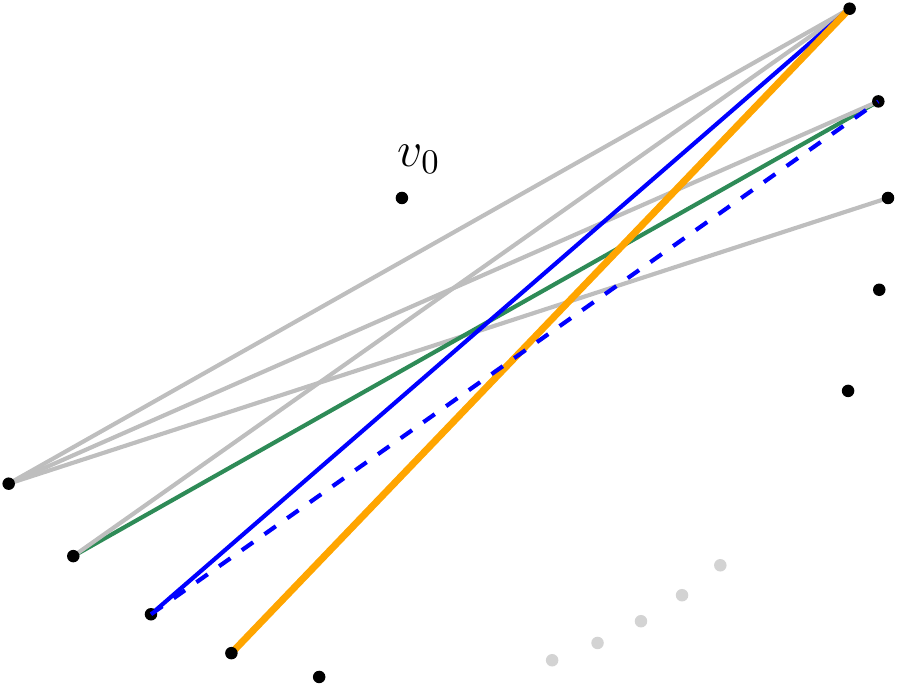}
\caption{\centering}
\label{fig:bw_k3_continuation_a}
\end{subfigure}\hspace{2cm}%
\begin{subfigure}{.3\textwidth}
\centering
\includegraphics[width=\textwidth,page=2]{bw_k3_continuation}
\caption{\centering}
\label{fig:bw_k3_continuation_b}
\end{subfigure}
\caption{All edges up to distance $d_{i-1} = d_3$ are covered (solid, non-orange). The orange edge of distance $d_i$ is already fixed. In (a) it is incident to one previous edge (blue), in (b) it is incident to two previous edges (red and blue). The dashed edges mark the forced continuations (of red and blue) and the green edge is the next in the circular order (whose orientation is then also fixed).}
\label{fig:bw_k3_continuation}
\end{figure}

Note that the properties of \Cref{lem:base_extension} mean that, in order to prove only the existence of \Cref{thm:bw_k3_classification}, we need to provide a base construction where we describe exactly the span of each spanning tree (and for $T_0$ all edges in $\cl(e^+)$ of its single maximal diagonal edge $e$). We, however, give a full characterization of all possible such partial partitions of $BW_{k,3}$ and show that there are \numPartitionsKThree non-isomorphic partitions into plane spanning trees in total.

Let us first show that there cannot be any other partitions, that is, any partition has to fulfill the properties of \Cref{lem:base_extension}:

\begin{lemma}
Let $T_0, \ldots, T_{n-1}$ be a partition of the edges of $BW_{k,3}$ into plane spanning trees. Then, there are subtrees $T'_0, \ldots, T'_{n-1}$ fulfilling the properties of \Cref{lem:base_extension}.
\end{lemma}

\begin{proof}
Define the subtrees $T'_i$ to consist only of the radial and maximal diagonal edges (of the respective $T_i$). Properties (\ref{item:base_extension_a}) and (\ref{item:base_extension_c}) are clearly fulfilled. Moreover, all $T_i$'s are non-empty and plane, and connectedness follows from the fact that the maximal diagonal edges define the span and within each span there can only be radial edges (\Cref{lem:bw_2_boundary}). Hence, property (\ref{item:base_extension_d}) holds as well. Finally, the first part of property (\ref{item:base_extension_b}) holds by \Cref{prop:one_max_edge_per_distance} and by \Cref{prop:preliminary2} there are no further maximal diagonal edges.
\end{proof}

So we have one maximal diagonal edge of distance $d_1$, $d_2$, and $d_3$ each between every pair of opposite groups. And we need to pair two of them in each tree (except for $T_0$, that only contains one such edge).
The following lemma, which we state in more general terms of subgraphs since we need it that way later (\Cref{sec:subgraphs}), gives a restriction on which edges we can pair.

\begin{lemma}\label{lem:plane_subgraphs1}
Let $e_1, \ldots, e_m$ be pairwise (non-crossing) incomparable edges of a plane subgraph of $BW_{k,\ell}$. Then, $\sum_i \dist(e_i) \leq 2n - 1\ (=k \ell)$ holds. Moreover, for \emph{two} (non-crossing) incomparable edges $e_1,e_2$
\[
\dist(e_1) + \dist(e_2) \leq 2n - 2
\]
holds.
\end{lemma}

\begin{proof}
The first part follows immediately from \Cref{obs:maximal_edges_disjoint_vertices} and the second part then from the fact that two edges cannot form a cycle.
\end{proof}

For convenience we include the following result.

\begin{lemma}\label{lem:plane_subgraphs_index_sums}
For any two incomparable edges with distances $d_i$ and $d_j$ ($1 \leq i,j < \frac{k+1}{2}\ell$) belonging to the same plane subgraph of $BW_{k,\ell}$, the inequality
\[
i + j \geq \ell + 1
\]
holds.
\end{lemma}

\begin{proof}
By \Cref{lem:plane_subgraphs1},
\[
\frac{k+1}{2}\ell - i + \frac{k+1}{2}\ell - j = d_i + d_j \leq 2n-2 = k\ell-1
\]
holds and rearranging terms yields the desired inequality.
\end{proof}

In our setting, this means that we can pair a maximal diagonal edge of distance $d_1$ only with one of distance $d_3$ in its plane spanning tree. And a maximal diagonal edge of distance $d_2$ we can only pair with one of distance $d_2$ or~$d_3$.

Further note that if the distances of the two maximal diagonal edges sum to $2n-2 = 3k - 1$ (so $d_1 + d_3$ or $d_2 + d_2$) the respective spanning tree has exactly one radial edge in its span, and if they only sum to $3k - 2$ ($d_2 + d_3$), there are exactly $2$ radial edges in the span.

In the following we call the radial edges incident to the middle vertex of each group \emph{center radial edges}. Also, we call the two outmost vertices of a group the \emph{right} and \emph{left} vertex of that group, respectively, as viewed from $v_0$ (that is, $v_{3i}$ is the right and $v_{3(i-1)+1}$ the left vertex of group $\mathcal{G}_i$). Finally, we are ready to prove \Cref{thm:bw_k3_classification}, which we restate for easier readability:

\thmBWKThreeClassification*

\begin{proof}
Using the structural properties derived on the way to proof \Cref{thm:disprove_conjecture_general}, we show how to construct all possible partial partitions fulfilling \Cref{lem:base_extension}. First of all, we need to have a tree $T_0$ using at least $(3\frac{k-1}{2} + 1)$ radial edges and a single maximal diagonal edge~$e$ (by \Cref{lem:preliminary1,prop:preliminary2}). Let $e$ again be between $\mathcal{G}_{\frac{k+1}{2}}$ and $\mathcal{G}_k$ and let the radial edge $\{v_0, v_{3\frac{k-1}{2}+1}\}$ be part of~$T_0$. 
If $\dist(e) = d_3$, there is one more maximal diagonal edge of distance $d_1$ than of distance $d_3$ left, so we cannot pair all edges of distance $d_1$. Further, if $\dist(e) = d_2$, all distance $d_1$ edges need to be paired with distance $d_3$ edges and, especially, all remaining distance $d_2$ edges need to be paired with each other. 
In particular, the tree containing the center radial edge of $\mathcal{G}_k$ (which cannot be part of $T_0$) would also have a maximal diagonal edge of distance $d_2$ between the groups $\mathcal{G}_{\frac{k+1}{2}}$ and $\mathcal{G}_k$; a contradiction to \Cref{prop:one_max_edge_per_distance}. Hence, we know $\dist(e) = d_1$. 

Considering the remaining $k-1$ maximal diagonal edges of distance $d_1$, they all must be paired with distance $d_3$ edges. This leaves one distance $d_3$ edge to be paired with a distance $d_2$ edge (let the respective tree be $T_1$) and two distance $d_2$ edges each for the remaining trees. In other words:

\begin{itemize}
\item $T_1$ contains one maximal diagonal edge of distance $d_2$ and $d_3$ each,
\item there are $\frac{k-1}{2}$ trees with two distance $d_2$ maximal diagonal edges (Class I), and
\item $k-1$ trees with one maximal diagonal edge of distance $d_1$ and $d_3$ each (Class II).
\end{itemize}

Next, consider the $\frac{k+1}{2}$ remaining center radial edges (not used by $T_0$), which clearly cannot be used by trees of Class II and no tree can use more than one of them. Hence, $T_1$ and every Class I tree needs to use exactly one such center radial edge. On the other hand, every Class I tree must have the center radial edge incident to its apex and hence, the center radial edges of groups $\mathcal{G}_{\frac{k+1}{2}}$ and $\mathcal{G}_k$ cannot both be used by Class I trees (again by contradiction to \Cref{prop:one_max_edge_per_distance} otherwise). Therefore, $T_1$ has to use one of them and some Class I tree $T_c$ the other. The remaining $\frac{k-3}{2}$ Class I trees use the center radial edges of the groups $\mathcal{G}_{\frac{k+3}{2}}$ to $\mathcal{G}_{k-1}$ (with apex in the respective group). 

\begin{figure}
\centering
\includegraphics[scale=0.55,page=4]{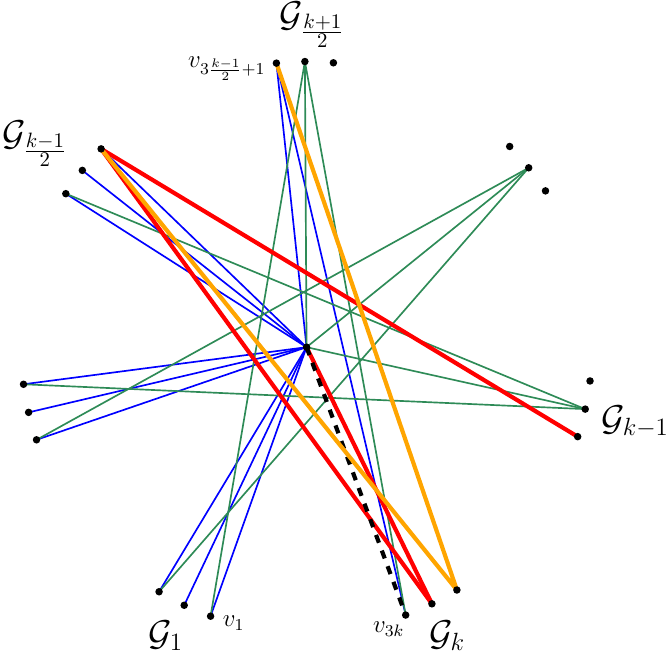}
\caption{Illustration of the claim in the proof of \Cref{thm:bw_k3_classification}. $T_0$ (blue) uses the $d_1$ edge between $\mathcal{G}_{\frac{k+1}{2}}$ and $\mathcal{G}_k$. If a Class I tree (green) uses the center radial edge of $\mathcal{G}_{\frac{k+1}{2}}$ (and therefore the $d_2$ edge between $\mathcal{G}_{\frac{k+1}{2}}$ and $\mathcal{G}_k$), either $T_1$ (red) or the orange tree uses the $d_3$ edge between $\mathcal{G}_{\frac{k+1}{2}}$ and $\mathcal{G}_k$. Consequently, the radial edge $\{v_0,v_{3k}\}$ (dashed black) cannot be accommodated anymore.}
\label{fig:bw_k3_all_base_constructions_0}
\end{figure}

\begin{claim*}
$T_c$ cannot use the center radial edge of $\mathcal{G}_{\frac{k+1}{2}}$.
\end{claim*}

\begin{claimproof}
Assume for the sake of contradiction that $T_c$ uses the center radial edge of $\mathcal{G}_{\frac{k+1}{2}}$ and hence, $T_1$ contains the center radial edge of $\mathcal{G}_k$. Moreover, $T_1$ has its $d_2$ edge between $\mathcal{G}_k$ and $\mathcal{G}_{\frac{k-1}{2}}$ (since all other $d_2$ edges are already taken by Class I trees). Then, $T_1$ has its $d_3$ edge either between $\mathcal{G}_k$ and $\mathcal{G}_{\frac{k+1}{2}}$, or between $\mathcal{G}_{\frac{k-1}{2}}$ and $\mathcal{G}_{k-1}$. In the latter case, the tree using the $d_1$ edge between $\mathcal{G}_{k}$ and $\mathcal{G}_{\frac{k-1}{2}}$ must have its $d_3$ edge between $\mathcal{G}_k$ and $\mathcal{G}_{\frac{k+1}{2}}$. In any case, the radial edge $\{v_0,v_{3k}\}$ cannot be accommodated anymore, since the corresponding tree would need to have a maximal diagonal edge between $\mathcal{G}_k$ and $\mathcal{G}_{\frac{k+1}{2}}$ (all of which are already taken). See \Cref{fig:bw_k3_all_base_constructions_0}.
\end{claimproof}

So, the Class I trees have their apexes in groups $\mathcal{G}_{\frac{k+3}{2}}$ to $\mathcal{G}_{k}$ and $T_1$ has to use the last remaining distance $d_2$ edge (between $\mathcal{G}_{\frac{k+1}{2}}$ and $\mathcal{G}_1$). Therefore, the distance $d_3$ edge of $T_1$ is either between $\mathcal{G}_{\frac{k+1}{2}}$ and $\mathcal{G}_k$ or between $\mathcal{G}_1$ and $\mathcal{G}_{\frac{k+3}{2}}$.

As a next step, consider the Class II trees. Remember that $T_0$ uses the distance $d_1$ edge between $\mathcal{G}_{\frac{k+1}{2}}$ and $\mathcal{G}_k$. Hence, the distance $d_3$ edge between $\mathcal{G}_1$ and $\mathcal{G}_{\frac{k+1}{2}}$ needs to be paired with the distance $d_1$ edge between $\mathcal{G}_1$ and $\mathcal{G}_{\frac{k+3}{2}}$, forcing the respective Class II tree to have its apex in the right vertex of $\mathcal{G}_1$. Subsequently, this forces another Class II tree to have its apex in the right vertex of $\mathcal{G}_2$ and so on, up to group $\mathcal{G}_{\frac{k-1}{2}}$.

Similarly, considering the distance $d_3$ edge between $\mathcal{G}_{\frac{k-1}{2}}$ and $\mathcal{G}_k$, we get a Class~II tree with its apex in the left vertex of $\mathcal{G}_{\frac{k-1}{2}}$. And again subsequently another with its apex in the left vertex of $\mathcal{G}_{\frac{k-3}{2}}$ and so on, down to group $\mathcal{G}_2$. Note that, since $T_1$ might use the distance $d_3$ edge between $\mathcal{G}_1$ and $\mathcal{G}_{\frac{k+3}{2}}$, we cannot yet conclude where the apex of the last Class~II tree will be (let that tree be $T_{n-1}$). So, let us summarize what we know so far:

\begin{itemize}
\item $T_0$ uses all radial edges $\{v_0, v_i\}$ (for $1 \leq i \leq 3\frac{k-1}{2} + 1$) and the maximal diagonal edge $\{v_{3\frac{k-1}{2} + 1}, v_{3k}\}$ of distance $d_1$ (between $\mathcal{G}_{\frac{k+1}{2}}$ and $\mathcal{G}_k$),
\item $T_1$ uses the center radial edge of group $\mathcal{G}_{\frac{k+1}{2}}$ and the maximal diagonal edge $\{v_1,v_{3\frac{k-1}{2} + 2}\}$ of distance $d_2$ (between $\mathcal{G}_1$ and $\mathcal{G}_{\frac{k+1}{2}}$),
\item the center radial edges of groups $\mathcal{G}_{\frac{k+3}{2}}$ to $\mathcal{G}_{k}$ are used by Class I trees (with apex in the respective group),
\item there are $k-2$ Class II trees with apexes in an outmost vertex from the right vertex of group $\mathcal{G}_{1}$ to the right vertex of group $\mathcal{G}_{\frac{k-1}{2}}$, and
\item the last Class II tree $T_{n-1}$ uses the maximal diagonal edge $\{v_1,v_{3\frac{k+1}{2}}\}$ of distance $d_1$ (between $\mathcal{G}_1$ and $\mathcal{G}_{\frac{k+1}{2}}$ --- it is the only distance $d_1$ edge left).
\end{itemize}

It remains to determine the distance $d_3$ edge and the second radial edge of $T_1$ as well as the distance $d_3$ edge (and hence the apex) of $T_{n-1}$. We will see that there are three base cases for that:

\begin{description}
\item[Case 1:] $T_1$ has its apex in $\mathcal{G}_1$.

That is, $T_1$ has its distance $d_3$ edge between $\mathcal{G}_1$ and $\mathcal{G}_{\frac{k+3}{2}}$. Hence, the second radial edge of $T_1$ is $\{v_0, v_{3\frac{k+1}{2}}\}$ and $T_{n-1}$ has its apex in the right vertex of $\mathcal{G}_{\frac{k+1}{2}}$ (see \Cref{fig:bw_k3_all_base_constructions_a}).

\item[Case 2:] $T_1$ has its apex in $\mathcal{G}_{\frac{k+1}{2}}$.

That is, $T_1$ has its distance $d_3$ edge between $\mathcal{G}_{\frac{k+1}{2}}$ and $\mathcal{G}_k$. Hence, $T_{n-1}$ has its distance $d_3$ edge between $\mathcal{G}_1$ and $\mathcal{G}_{\frac{k+3}{2}}$ (so its apex is in the left vertex of $\mathcal{G}_1$) and therefore the second radial edge of $T_1$ must be $\{v_0, v_{3k}\}$ (since no other tree can use it anymore). However, there are two possibilities for the distance $d_3$ edge of $T_1$ now:
\begin{enumerate}[a)]
\item either $\{v_{3\frac{k+1}{2}}, v_{3k}\}$ as depicted in \Cref{fig:bw_k3_all_base_constructions_b}, or
\item $\{v_{3\frac{k-1}{2}+2}, v_{3(k-1)+2}\}$ as depicted in \Cref{fig:bw_k3_all_base_constructions_c}.
\end{enumerate}
\end{description}

\begin{figure}
\centering
\begin{subfigure}{.25\textwidth}
\centering
\includegraphics[width=\textwidth,page=1]{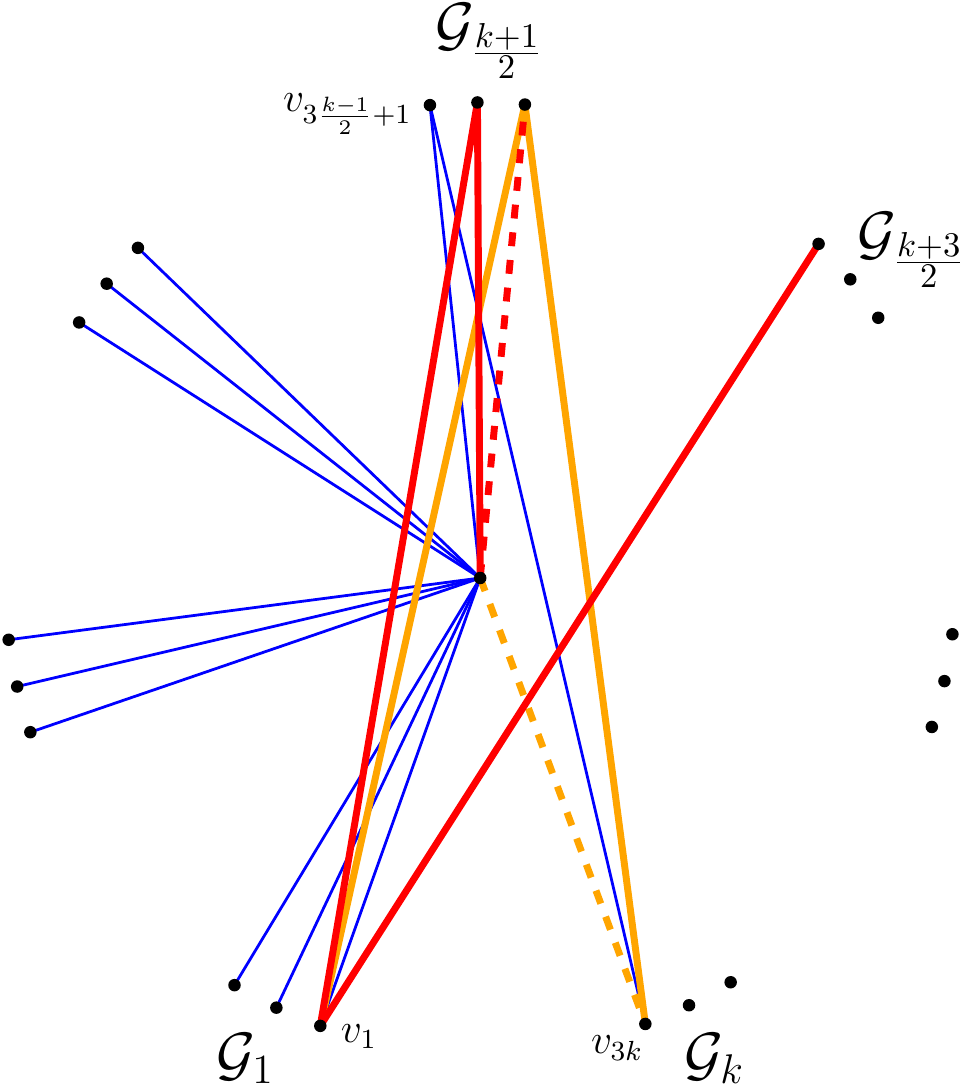}
\caption{\centering Case 1}
\label{fig:bw_k3_all_base_constructions_a}
\end{subfigure}\hspace{.1\textwidth}%
\begin{subfigure}{.25\textwidth}
\centering
\includegraphics[width=\textwidth,page=2]{bw_k3_all_base_constructions}
\caption{\centering Case 2a}
\label{fig:bw_k3_all_base_constructions_b}
\end{subfigure}\hspace{.1\textwidth}%
\begin{subfigure}{.25\textwidth}
\centering
\includegraphics[width=\textwidth,page=3]{bw_k3_all_base_constructions}
\caption{\centering Case 2b}
\label{fig:bw_k3_all_base_constructions_c}
\end{subfigure}
\caption{Illustration of the three base cases. $T_0$ is depicted in blue, $T_1$ is red, and $T_{n-1}$ orange. The red dashed radial edge depicts the second radial edge, that $T_1$ has to use.}
\label{fig:bw_k3_all_base_constructions}
\end{figure}

By now, we have fixed all maximal diagonal edges and they fulfill property (\ref{item:base_extension_b}) of \Cref{lem:base_extension}. Finally, it only remains to determine the radial edges of all Class II trees. Note that for each group $\mathcal{G}'$ (from $\mathcal{G}_1$ to $\mathcal{G}_{\frac{k-1}{2}}$) with two apexes of Class II trees in it (at the outmost vertices), those two trees have the same pair of radial edges in their span (in the groups opposite to~$\mathcal{G}'$). So we have two (non-isomorphic) choices for each of those cases (independently).

Further, in Case 1 $T_{n-1}$ needs to contain $\{v_0, v_{3k}\}$ and the Class II tree with apex in the right vertex of $\mathcal{G}_1$ needs to contain $\{v_0, v_{3\frac{k+1}{2}+1}\}$ (the other radial edges in its span are used by $T_1$ and $T_0$). This leaves a pair of Class II trees for each group from $\mathcal{G}_2$ to $\mathcal{G}_{\frac{k-1}{2}}$. Meanwhile, in Case 2 we have a pair of Class II trees for each group from $\mathcal{G}_1$ to $\mathcal{G}_{\frac{k-1}{2}}$. That is, in Case 2 we have exactly $2^{\frac{k-1}{2}}$ possibilities (for the Class II trees) to form a partial partition fulfilling the properties of \Cref{lem:base_extension}, while in Case 1 we only have $2^{\frac{k-3}{2}}$ possibilities.

Finally, by \Cref{lem:full_extension}, there are $2^{3\frac{k-1}{2} - 1}$ possible extensions to get a full partition from such a partial one. Hence, in total we get
\[
2^{3\frac{k-1}{2} - 1}(2\cdot 2^{\frac{k-1}{2}} + 2^{\frac{k-3}{2}}) = 4^{k-1} + 4^{k-2}
\]
non-isomorphic partitions.
\end{proof}

In \Cref{app:sec:additional_figures}, we give some examples of partitions corresponding to the three base constructions.

\subsection{Partition into Plane Subgraphs}\label{sec:subgraphs}

In the previous section, we gave a classification which bumpy wheels can be partitioned into plane spanning trees and which not. Surprisingly it turns out that allowing arbitrary plane subgraphs does not help much. The only bumpy wheel that can be partitioned into plane subgraphs but not into plane spanning trees is $BW_{3,5}$. 

Note that before we also heavily exploited the structure enforced by spanning trees. This is not possible anymore for the case of arbitrary plane subgraphs. We cannot make any assumptions on the number of edges, not even about connectedness. The only property we can (and will) exploit is the fact, that we still have maximal diagonal edges and radial edges may only be contained in their span (recall \Cref{lem:bw_2_boundary}).

We split the proof of \Cref{thm:main_subgraphs} into two parts, first focusing on the case $\ell > 5$.

\begin{restatable}{theorem}{disproveSubgraphs}\label{thm:disprove_subgraphs}
For any odd parameters $k \geq 3$ and $\ell > 5$, the edges of $BW_{k,\ell}$ cannot be partitioned into $n = \frac{k\ell + 1}{2}$ plane subgraphs.
\end{restatable}

The proof is more technical than for spanning trees. Again, we start with some structural results. Recall that the distances of two incomparable edges sum to at most $2n-2$ (\Cref{lem:plane_subgraphs1}). Also recall the definition of distance $d_i = \frac{k+1}{2}\ell - i$ from \Cref{sec:tree_disprove}.

\begin{proposition}\label{prop:plane_subgraphs1a}
Let $D_0, \ldots, D_{n-1}$ be a partition of $BW_{k,\ell}$ into $n = \frac{k\ell + 1}{2}$ plane subgraphs (if it exists). Then between each pair of opposite groups and for each $1 \leq i \leq \ell$ there are at least $i$ diagonal edges of distance at least $d_i$ that are maximal in their subgraph. %
\end{proposition}

\begin{proof}
First note that $\dist(e) \leq d_1$ holds for any diagonal edge $e$. Furthermore, any diagonal edge $e$ with $\dist(e) \geq d_\ell$ connects opposite groups.

Observe now that between any pair of opposite groups and any $1 \leq i \leq \ell$ there is a crossing family of size $i$ (where each edge has distance $d_i$). Since all these edges need to belong to different subgraphs and larger edges (with respect to $<_c$) may only exist between the same pair of opposite groups, there need to be at least $i$ maximal diagonal edges of distance at least $d_i$ between each pair of opposite groups.
\end{proof}

Rephrasing \Cref{prop:plane_subgraphs1a}, we may also say that between each pair of opposite groups there is one maximal diagonal edge of distance (at least) $d_1$, \emph{another} of distance at least $d_2$, yet \emph{another} of distance at least $d_3$, and so on until distance~$d_\ell$.

And since there are exactly $k$ pairs of opposite groups, \Cref{prop:plane_subgraphs1a} guarantees at least $k \cdot \ell$ maximal diagonal edges of distance at least $d_\ell$ in any partition of $BW_{k,\ell}$. For each pair of opposite groups, we distinctly pick for each  $1 \leq i \leq \ell$ one of those edges that has distance at least $d_i$ to get precisely $k \cdot \ell$ edges in total, which we call \emph{forced diagonal} edges in the following (or forced edges for short).

Let $E_{\text{forced}}$ be the set of forced diagonal edges, then from the definition it follows:
\begin{equation}\label{eq:forced_lower_bound}
\sum_{e \in E_{\text{forced}}} \dist(e) \geq k\sum_{i=1}^\ell \left(\frac{k+1}{2}\ell - i\right)
\end{equation}

Note in the following that \Cref{lem:plane_subgraphs_index_sums,lem:plane_subgraphs1} hold especially for forced diagonal edges (contained in the same plane subgraph), since they are maximal and therefore incomparable.

\begin{proposition}\label{prop:plane_subgraphs2}
Let $D_0, \ldots, D_{n-1}$ be a partition of $BW_{k,\ell}$ into $n = \frac{k\ell + 1}{2}$ plane subgraphs (if it exists). Then one subgraph, say $D_0$, contains exactly one forced diagonal edge and all other $n-1$ subgraphs contain exactly two forced diagonal edges.
\end{proposition}

\begin{proof}
In total there are $k\cdot \ell$ forced edges and we have $n = \frac{k\ell + 1}{2}$ subgraphs. Hence, to prove the statement, we only need to show that no subgraph can contain more than two forced edges. To this end, we consider the two cases $k=3$ and $k>3$ separately:

\begin{description}
\item Case 1: $k > 3$.

By \Cref{lem:plane_subgraphs1}, in each subgraph the sum of distances of its maximal diagonal edges may not exceed $k\ell$. If there is a subgraph with more than two forced edges (each of distance at least $d_\ell = \frac{k-1}{2}\ell$), the sum of their distances is at least $3\cdot \frac{k-1}{2}\ell$, which is $> k\ell$ for $k>3$. 

For the case $k=3$, we need a more careful analysis.

\item Case 2: $k=3$.

If there is a subgraph $D'$ containing three forced edges, then all three need to have distance exactly $d_\ell = \ell$ (for a combined distance of at most $3\ell$) and the subgraph cannot contain any further forced edges. Consider then the three forced edges of distance $d_1$. By \Cref{lem:plane_subgraphs_index_sums} they can only be paired with forced edges of distance at most $d_\ell$. However, the three forced edges of distance (at least) $d_\ell$ are taken by $D'$. So we get three subgraphs containing only one forced edge (of distance $d_1$). Together with $D'$ this leaves $3\ell-6$ forced edges of distance at least $d_{\ell-1}$ to be covered by the remaining subgraphs (note that by \Cref{lem:plane_subgraphs1} any of the other subgraphs can contain at most two of them). Hence, in total we would already need at least 
\[
4 + \frac{3\ell - 6}{2} = \frac{3\ell + 1}{2} + 0.5 = n + 0.5
\]
plane subgraphs to cover all forced edges.
\end{description}

Therefore, we need $n-1$ subgraphs containing exactly two forced edges and one subgraph (say $D_0$) containing one forced edge, to cover all $k \cdot \ell = 2n-1$ forced edges.
\end{proof}

The reader might notice that \Cref{prop:plane_subgraphs2} is similar to \Cref{prop:preliminary2} for plane spanning trees (in \Cref{sec:tree_disprove}) and it also holds for those. Furthermore, the following proposition is a generalization of an idea, that we already used towards the proof of \Cref{thm:bw_k3_classification}.

\begin{proposition}\label{prop:plane_subgraphs2a}
Let $D_0, \ldots, D_{n-1}$ be a partition of $BW_{k,\ell}$ into $n = \frac{k\ell + 1}{2}$ plane subgraphs (if it exists) and let $e$ be the forced diagonal edge in $D_0$ with $\dist(e) = d_1 - x_0$ (for some $x_0 \geq 0$). Moreover, for each $1 \leq i \leq n-1$, let $x_i \ge 0$ be such that the distances of the two forced diagonal edges $e_i$ and $e_i'$ in $D_i$ sum to $2n-2-x_i$.
Then, 
\[
\sum_{i=0}^{n-1} x_i \leq \frac{\ell - 1}{2}
\]
holds.
\end{proposition}

\begin{proof}
By \Cref{eq:forced_lower_bound}, we know that
\[
\dist(e) + \sum_{i=1}^{n-1} \left( \dist(e_i) + \dist(e_i') \right) \geq k\sum_{i=1}^{\ell} \left( \frac{k+1}{2}\ell - i \right) = k \cdot (k\ell-1) \frac{\ell}{2}
\]
holds for the sum of all forced diagonal edges. Plugging in $\dist(e) = \frac{k + 1}{2}\ell - 1 - x_0$ and $\dist(e_i) + \dist(e_i') = k\ell - 1 - x_i$, yields
\[
\frac{k + 1}{2}\ell - 1 + (n-1)(k\ell - 1) - \sum_{i=0}^{n-1} x_i \geq k\ell \frac{k\ell-1}{2}.
\]
Finally, plugging in $n = \frac{k\ell+1}{2}$ and rearranging terms
\[
\sum_{i=0}^{n-1} x_i \leq \frac{k + 1}{2}\ell - 1 + \frac{k\ell-1}{2}(k\ell - 1) - k\ell \frac{k\ell-1}{2} = \frac{\ell - 1}{2}
\]
provides the desired result.
\end{proof}

Let the forced diagonal edge $e$ in $D_0$ from now on connect groups $\mathcal{G}_\frac{k+1}{2}$ and $\mathcal{G}_k$ (similar to the maximal diagonal edge of $T_0$ in \Cref{sec:tree_disprove}). Further note that two endpoints of the two forced diagonal edges $e_i$ and $e_i'$ in $D_i$ ($i \neq 0$) must be contained in a common group $\mathcal{G}$, since otherwise $e_i$ and $e_i'$ would cross. We call the set of vertices in $\spn(e_i,e_i')$ that lie in this common group $\mathcal{G}$ the \emph{apex} of $D_i$ (again similar as in \Cref{sec:tree_disprove}). Especially note that all radial edges in $D_i$ ($i \neq 0$) must be contained in $\spn(e_i,e_i')$ (similar to \Cref{lem:bw_2_boundary}).

In the following description of \emph{additional} vertices we do not count $v_0$. Note that $\cl(e^+)$ contains exactly $(\frac{k-1}{2}\ell + 2 + x_0)$ vertices (those of the first $\frac{k-1}{2}$ groups, two outmost vertices, and $x_0$ extra vertices). Moreover, $\spn(e_i,e_i')$ ($i \neq 0$) contains exactly $(2 + 1 + x_i)$ vertices (two outmost vertices, one vertex in the apex, and $x_i$ extra vertices (possibly also in the apex, making its size larger than $1$)). We call the set of all extra vertices 
\emph{additional} vertices.

\begin{lemma}\label{lem:plane_subgraphs_inside_radial}
In any partition of $BW_{k,\ell}$ into $n = \frac{k\ell + 1}{2}$ plane subgraphs any inside radial edge of the last $\frac{k+1}{2}$ groups must either be incident to the apex or to an additional vertex of its subgraph.
\end{lemma}

\begin{proof}
Let $f$ be an inside radial edge of one of the last $\frac{k+1}{2}$ groups. If $f$ belongs to $D_0$, it is incident to an additional vertex (since it is not in the first $\frac{k-1}{2}$ groups and not incident to one of the two outmost vertices). On the other hand, if $f$ belongs to some $D_i$ with $i \neq 0$, it still cannot be incident to one of the two outmost vertices in $\spn(e_i,e_i')$, so it is either incident to the apex or an additional vertex.
\end{proof}

Finally we call the inside radial edges and inside vertices of the two groups $\mathcal{G}_\frac{k+1}{2}$ and $\mathcal{G}_k$ \emph{special} radial edges and \emph{special} vertices. Now we are ready to prove \Cref{thm:disprove_subgraphs}, which we again restate for easier readability:

\disproveSubgraphs*

\begin{proof}
Note that there are exactly $2\ell - 4$ special radial edges. By \Cref{lem:plane_subgraphs_inside_radial} they must either be incident to an apex (in one of the two groups) or an additional vertex. Any subgraph with an apex in $\mathcal{G}_\frac{k+1}{2}$ or $\mathcal{G}_k$ must contain a forced edge between those two groups. By the definition of forced edges there are exactly $\ell$ of them between $\mathcal{G}_\frac{k+1}{2}$ and $\mathcal{G}_k$ and one of them is taken by $D_0$. So there are at most $\ell - 1$ apexes in $\mathcal{G}_\frac{k+1}{2}$ and $\mathcal{G}_k$ together. Finally, \Cref{prop:plane_subgraphs2a} gives an upper bound on the total number of additional vertices.
Hence, in total we can cover at most
\[
(\ell - 1) + \frac{\ell-1}{2} = \frac{3}{2}(\ell - 1)
\]
special radial edges with our $n$ subgraphs (note that we count only one vertex for each apex because if an apex is larger than 1, that is, contains additional vertices, we already accounted for that in the additional vertices bound). Therefore, whenever $\frac{3}{2}(\ell - 1) < 2\ell - 4$ holds, we cannot cover all edges. This inequality is equivalent to $\ell > 5$.
\end{proof}

For the case $\ell = 5$, we need to go even deeper into the structure of our plane subgraphs. 

\begin{restatable}{theorem}{disproveSubgraphsKFive}\label{thm:disprove_subgraphs_k5}
For any odd parameter $k \geq 5$, the edges of $BW_{k,5}$ cannot be partitioned into $n = \frac{5k + 1}{2}$ plane subgraphs.
\end{restatable}

\begin{proof}
As before, we first consider the special radial edges, that is, the inside radial edges of the groups $\mathcal{G}_\frac{k+1}{2}$ and $\mathcal{G}_k$. Since $\ell = 5$, there are $6$ of them here. Furthermore, any subgraph that has its apex in $\mathcal{G}_\frac{k+1}{2}$ or $\mathcal{G}_k$ must use a forced edge between those two groups (the blue stripe in \Cref{app:fig:bw_k5}). However, because $D_0$ already uses one of these $5$ forced edges, there are at most $4$ apexes in $\mathcal{G}_\frac{k+1}{2}$ and $\mathcal{G}_k$ together and (by \Cref{prop:plane_subgraphs2a}) at most $\frac{5-1}{2} = 2$ additional vertices in total.

So in order to cover all special radial edges, all additional vertices (two in our case) must be special vertices. Furthermore, all forced edges must attain their minimal possible distance (recall that \Cref{prop:plane_subgraphs1a} was only guaranteeing a lower bound), since otherwise the bound in \Cref{prop:plane_subgraphs2a} concerning the additional vertices would get smaller. Hence, we have exactly one forced edge of each distance $d_1, \ldots, d_5$ between each pair of opposite groups.

Consider now the inside radial edges in all groups from $\mathcal{G}_\frac{k+3}{2}$ to $\mathcal{G}_{k-1}$. By \Cref{lem:plane_subgraphs_inside_radial} and the fact that all additional vertices are special vertices, they must be incident to some apex. Also, since the opposite groups are between $\mathcal{G}_1$ and $\mathcal{G}_\frac{k-1}{2}$, it is not possible to place special vertices as additional vertices in the respective subgraphs. Hence, these subgraphs use up all forced edges of distances $d_2, d_3, d_4$ (in the black stripes in \Cref{app:fig:bw_k5}), except of course those between the pairs of opposite groups $\mathcal{G}_1$ and $\mathcal{G}_\frac{k+1}{2}$, $\mathcal{G}_\frac{k+1}{2}$ and $\mathcal{G}_k$, and $\mathcal{G}_k$ and~$\mathcal{G}_\frac{k-1}{2}$.

Since $k \geq 5$, each edge between $\mathcal{G}_1$ and $\mathcal{G}_\frac{k+1}{2}$ crosses every edge between $\mathcal{G}_k$ and $\mathcal{G}_\frac{k-1}{2}$ (the red stripes in \Cref{app:fig:bw_k5}). Furthermore, two forced edges between the same pair of opposite groups cannot be in the same subgraph (because they are both maximal). Hence, we have $6$ forced edges of distances $d_2,d_3,d_4$ (call them \emph{leftover} edges) between those two pairs of opposite groups, that we still need to pair with a second forced edge in their subgraph.

However, by \Cref{lem:plane_subgraphs_index_sums} all forced edges of distance $d_1$ (except one if used by $D_0$) need to be paired with a distance $d_5$ forced edge in their subgraph. This leaves the $3$ forced edges of distances $d_2, d_3, d_4$ between groups $\mathcal{G}_\frac{k+1}{2}$ and $\mathcal{G}_k$, and possibly one forced edge of distance $d_5$ to pair the leftover edges with. That is two less than what we would need.
\end{proof}

\begin{figure}
\centering
\includegraphics[scale=0.5,page=1]{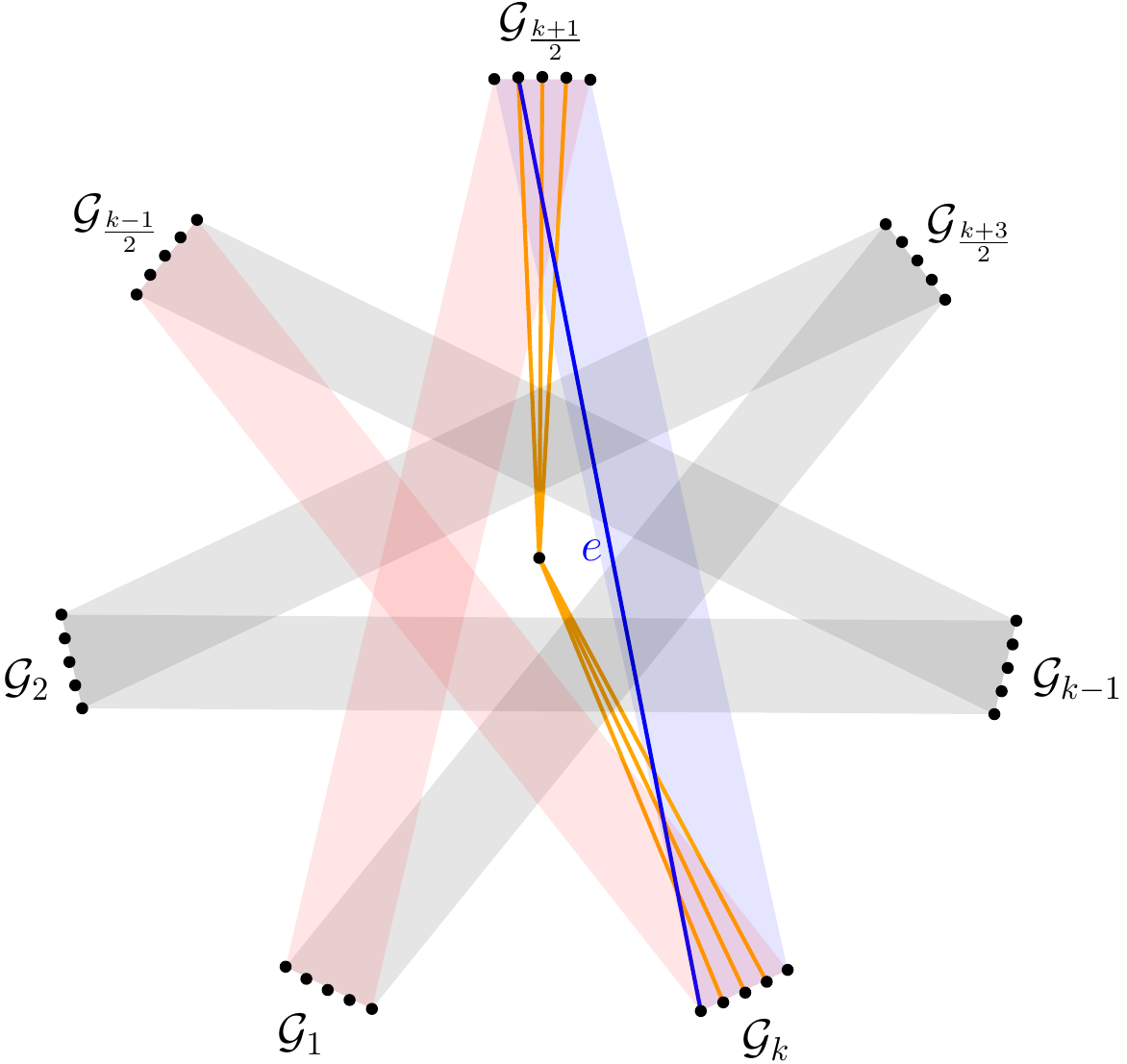}
\caption{High level overview of the proof of \Cref{thm:disprove_subgraphs_k5}. We have at most $\frac{5-1}{2} = 2$ additional vertices in total and the blue stripe (which contains the single forced edge $e$ of $D_0$) has to use both of them. Then, in the black stripes we must use all forced edges of distances $d_2, d_3, d_4$. However, since the two red stripes intersect ($k \geq 5$), there will not be enough forced edges to pair all 6 forced edges of distances $d_2, d_3, d_4$ from the red stripes.}
\label{app:fig:bw_k5}
\end{figure}

Finally, using \Cref{thm:main_spanning_trees}, it only remains to show that there is a partition for $BW_{3,5}$. However, in \Cref{sec:ILP} (see \Cref{fig:bw_35}) we already gave such a partition.

\section{Generalized Wheels}\label{sec:generalized_bw}

In this section we generalize our construction to non-regular wheel sets. We give a sufficient condition in the setting of plane spanning trees and a full characterization for partitioning into plane double stars. Recall that, for $N = [n_1, \ldots, n_k]$ and integers $n_i \geq 1$, $GW_N$ denotes the generalized wheel with group sizes $n_i$. Let us restate the result for the sake of readability:

\generalizedBW*

Note that the geometric regularity of generalized wheels is not strictly required (but eases the proofs). In fact, one can show that for any point set~$P$ (in general position) with exactly one point inside its convex hull, there is a generalized wheel with the exact same set of crossing edge pairs (and therefore the exact same set of plane subgraphs), 
as we show in \Cref{app:sec:drop_geometry}.

The proof of \Cref{thm:generalized_bw} is analogous to that of \Cref{thm:disprove_conjecture_general}. We point out the differences in this section. Note here, that the condition in \Cref{thm:generalized_bw} is only a sufficient condition but not a necessary one. In fact, we found some generalized wheels not fulfilling the condition, that still cannot be partitioned into plane spanning trees (verified by computer assistance), for example, $GW_{[2,3,3,4,5]}$ cannot be partitioned.

Almost all results from \Cref{obs:maximal_edges_disjoint_vertices} to \Cref{prop:one_max_edge_per_distance} hold for generalized wheels as well, exactly the way they were stated for bumpy wheels.

One exception is that \Cref{lem:preliminary1} (\ref{lem:prel1:c}) only guarantees the presence of all radial edges along $\frac{k-1}{2}$ consecutive groups plus one in $T_0$. That is, we need to drop the exact bound on the total number of radial edges in $T_0$; which we only used to show (in \Cref{prop:preliminary2}) that there is only a single spanning tree in the partition with exactly one boundary edge. However, the only way there can be two such trees in the generalized setting is, when one group has size at least $n$; which means the condition of \Cref{thm:generalized_bw} is not fulfilled anyway. In other words, even though \Cref{lem:preliminary1} (\ref{lem:prel1:c}) changed, we can still apply it for the proof of the generalized setting as before.

Furthermore, \Cref{prop:one_max_edge_per_distance} does not guarantee a unique maximal diagonal edge (between opposite groups $\mathcal{G}_x, \mathcal{G}_y$) per distance anymore, but only \emph{at most} one (which is actually all we need). Also note that the value of the distances may now differ (depending on the group sizes); though only the number of different distances is important here, which is $\max(n_x, n_y)$.

Let us also briefly point out the difference in the proof of \Cref{prop:one_max_edge_per_distance}: we do not have exactly $j$ edges of distance $d_j$ between opposite groups but only at most $j$ (since some of them might be between other groups now). The arguments in the proof still apply however.

For convenience we introduce a little more notation. For $j \in \{1,\ldots, k\}$, define (the indices of) the $\frac{k-1}{2}$ consecutive groups starting at $\mathcal{G}_j$ as $I_j$, that is, $I_j = \{j, j+1, \ldots, j+\frac{k-1}{2} - 1\}$ (as usual all indices are taken modulo~$k$). Then, \Cref{thm:generalized_bw} can be equivalently phrased as follows: $GW_N$ cannot be partitioned into plane spanning trees if for all $j \in \{1,\ldots, k\}$ the inequality $\sum_{i\in I_j} n_i < n - 2$ holds.

\begin{proof}[Proof of \Cref{thm:generalized_bw}]
First of all note that if there is an $i$ such that $n_i = 1$ (that is, a group consisting of only one vertex $v_1$), the condition of \Cref{thm:generalized_bw} cannot be satisfied: Consider the line through $v_0$ and $v_1$, then on each side there are $\frac{k-1}{2}$ consecutive groups and one side must contain a total of at least $n-1$ vertices. So, we can assume from now on that $n_i \geq 2$ for all $1 \leq i \leq k$.

The remainder of the proof is analogous to the proof of \Cref{thm:disprove_conjecture_general}. Let the condition in \Cref{thm:generalized_bw} be satisfied and assume for the contrary there was a partition $T_0, \ldots, T_{n-1}$. Again, we can have at most $k+1$ spanning trees containing an outmost radial edge.

Counting the trees not containing any outmost radial edge, we need to be a bit more careful now. We claim, that there are at most
\[ 
\sum_{i\in I_j}(n_i - 2)
\]
spanning trees not containing any outmost radial edge (for the $j$ maximizing this value). The arguments are exactly as before: each group $\mathcal{G}_i$ can contain the apex of at most $n_i - 2$ special wedges. And again by \Cref{prop:one_max_edge_per_distance}, instead of $\frac{k+1}{2}$ (by \Cref{lem:preliminary1} (\ref{lem:prel1:c}) the radial edges of at least $\frac{k-1}{2}$ consecutive groups are still taken by $T_0$; we just need to consider all possible choices now, because the groups have different sizes) we may sum only $\frac{k-1}{2}$ consecutive groups.

Hence, whenever
\[
k + 1 + \sum_{i \in I_j} (n_i - 2) < n
\]
holds for all $j$, we cannot find enough spanning trees. Rearranging terms, this inequality is equivalent to $\sum_{i \in I_j} n_i < n - 2$ (recall that $\lvert I_j \rvert = \frac{k-1}{2}$).
\end{proof}

\subparagraph*{Plane Double Stars.} Considering the other side of the story, \Cref{thm:double_stars} shows that many generalized wheels can already be partitioned into plane double stars\footnote{All double stars in this section are spanning (which we may not always spell out for easier readability).}:

\thmDoubleStars*

The proof requires several tools introduced in \cite{schnider2016packing}. In a first step we identify conditions under which a point set admits a so-called \emph{spine matching} --- the collection of spine edges from a partition into double stars (\Cref{app:thm:spine_matching}). Using these conditions we show (in \Cref{lem:bad_halfplanes}) that a generalized wheel $GW_N$ cannot be partitioned into plane double stars if and only if $GW_N$ has three \emph{bad halfplanes} whose intersection is empty (for a non-radial \emph{halving} edge $e$, the closure of $e^-$ defines a \emph{bad halfplane}). All details can be found in \Cref{app:sec:double_stars}.

We phrased \Cref{thm:double_stars} this way to make it consistent with \Cref{thm:generalized_bw}; however, let us rephrase it in a way that might better indicate the gap between the two theorems. To this end, define $F_i$ to be the \emph{family} of $\frac{k-1}{2}$ consecutive groups starting at $\mathcal{G}_i$ (in clockwise order). Two families $F_i$ and $F_{i+1}$ are called \emph{consecutive} and $|F_i|$ denotes the number of vertices in~$F_i$. If $|F_i| \leq n-2$ holds, we call $F_i$ \emph{small}, and otherwise \emph{large}.

\begin{corollary}
Let $GW_N$ be a generalized wheel with $k$ groups and $2n$ vertices. Then $GW_N$ \emph{can} be partitioned into plane spanning double stars if and only if there are $\frac{k-1}{2}$ consecutive families each containing (strictly) more than $n-2$ vertices.
\end{corollary}

\begin{proof}
If, for the one direction, there are $\frac{k-1}{2}$ large consecutive families, then there is a group $\mathcal{G}^\star$ (namely the one that is contained in all these $\frac{k-1}{2}$ families) such that any family containing $\mathcal{G}^\star$ is large. In particular, there cannot be three small families covering all groups. Hence, by \Cref{thm:double_stars}, there is a partition into plane double stars.

On the other hand, if there are no $\frac{k-1}{2}$ large consecutive families, we can find three small families as follows. Note first that every group is contained in some small family. Pick a small family $F$ arbitrarily and let $\mathcal{G}$ be the first group after $F$ (in clockwise order). Among all small families containing $\mathcal{G}$, pick the one that is \enquote{furthest} from $F$, that is, has least overlap with $F$, and call it $F'$. Let $\mathcal{G}'$ again be the first group after $F'$ and among all small families containing $\mathcal{G}'$ pick the one furthest from $F'$ and call it $F''$. Since $F''$ cannot contain $\mathcal{G}$, we conclude that the three small families $F, F', F''$ cover all groups.
\end{proof}

\section{Conclusion}

We showed that there are complete geometric graphs that cannot be partitioned into plane spanning trees and gave a full characterization of partitionability for bumpy wheels (even in the much broader setting of partitioning into plane subgraphs). Also, for generalized wheel sets we gave sufficient and necessary conditions. There are two obvious directions for further research in this setting, that is, on the one hand one could try to further classify which point sets can be partitioned and which cannot (this might also be a useful approach towards attacking the question concerning the complexity of the decision problem whether a given complete geometric graph admits a partition into plane spanning trees). On the other hand, one can study partitions into broader classes of subgraphs, for example, $k$-planar or $k$-\quasiplanar subgraphs.

The following very intriguing question to determine \emph{how far} we may get from the $|P|/2$ bound is still open:

\begin{question}[\cite{ferran_2006}]
Can any complete geometric graph on $n$ vertices be partitioned into $\frac{n}{c}$ plane subgraphs for some constant $c > 1$?
\end{question}



\bibliography{../../partition_literature}

\appendix

\section{The ILP model}\label{sec:ILP}

Given a  geometric graph $G = (P, E)$ and a fixed number $m$ of available colors as input, our ILP contains a binary variable~$x_{e,c} \in \{0,1\}$ for each edge-color combination, that is, in our setting there are $\binom{2n}{2}\cdot m$ variables. A variable $x_{e,c}$ being 1 then corresponds to edge $e$ receiving color~$c$.

We implement the following constraints, enforcing that (\ref{constraint_1}) every edge receives exactly one color, (\ref{constraint_2}) crossing edges receive different colors, (\ref{constraint_3}) ensuring $2n-1$ edges in each color class, and (\ref{constraint_4}) forbidding monochromatic triangles (clearly, these are necessary but not sufficient constraints):

\setcounter{equation}{0}
\begin{align}
\label{constraint_1} 
\sum_{c=1}^m x_{e,c} = 1 \qquad &\forall e\in E \\[4pt] \label{constraint_2}
x_{e,c} + x_{f,c} \leq 1 \qquad &\forall c\in \{1,\ldots,m\};\  \forall e,f \text{ crossing} \\[4pt] \label{constraint_3}
\sum_{e \in E} x_{e,c} = 2n-1 \qquad &\forall c\in \{1,\ldots,m\} \\[4pt] \label{constraint_4}
x_{e,c} + x_{f,c} + x_{g,c} \leq 2 \qquad &\text{for each triangle } e,f,g;\ \forall c\in \{1,\ldots,m\}
\end{align}

Any input that cannot satisfy these constraints, can also not be partitioned into plane spanning trees. For $BW_{3,5}$ and $m=8$ as input, using an industry shaped ILP solver, the ILP turns out to be infeasible (taking less than a minute). Furthermore, for $BW_{3,7}$ and $m=11$ as input our program reports an infeasible ILP even when omitting the constraints~(\ref{constraint_3}) and~(\ref{constraint_4}) (taking roughly 5h). \Cref{fig:bw_35} shows a partition of $BW_{3,5}$ into plane subgraphs found by the program, when omitting the triangle constraint~(\ref{constraint_4}).

All experiments were run on an Intel Core i5, 1.6 GHz, 16 GB RAM running macOS Big Sur Version~11.4. All algorithms were implemented in Python 3.9.1, and for solving the ILP we used Gurobi Optimizer Version 9.1.2 with default settings. 

\begin{figure}[htb]
\captionsetup[subfigure]{labelformat=empty}
\centering
\begin{subfigure}{.33\textwidth}
\centering
\includegraphics[width=\textwidth,page=1]{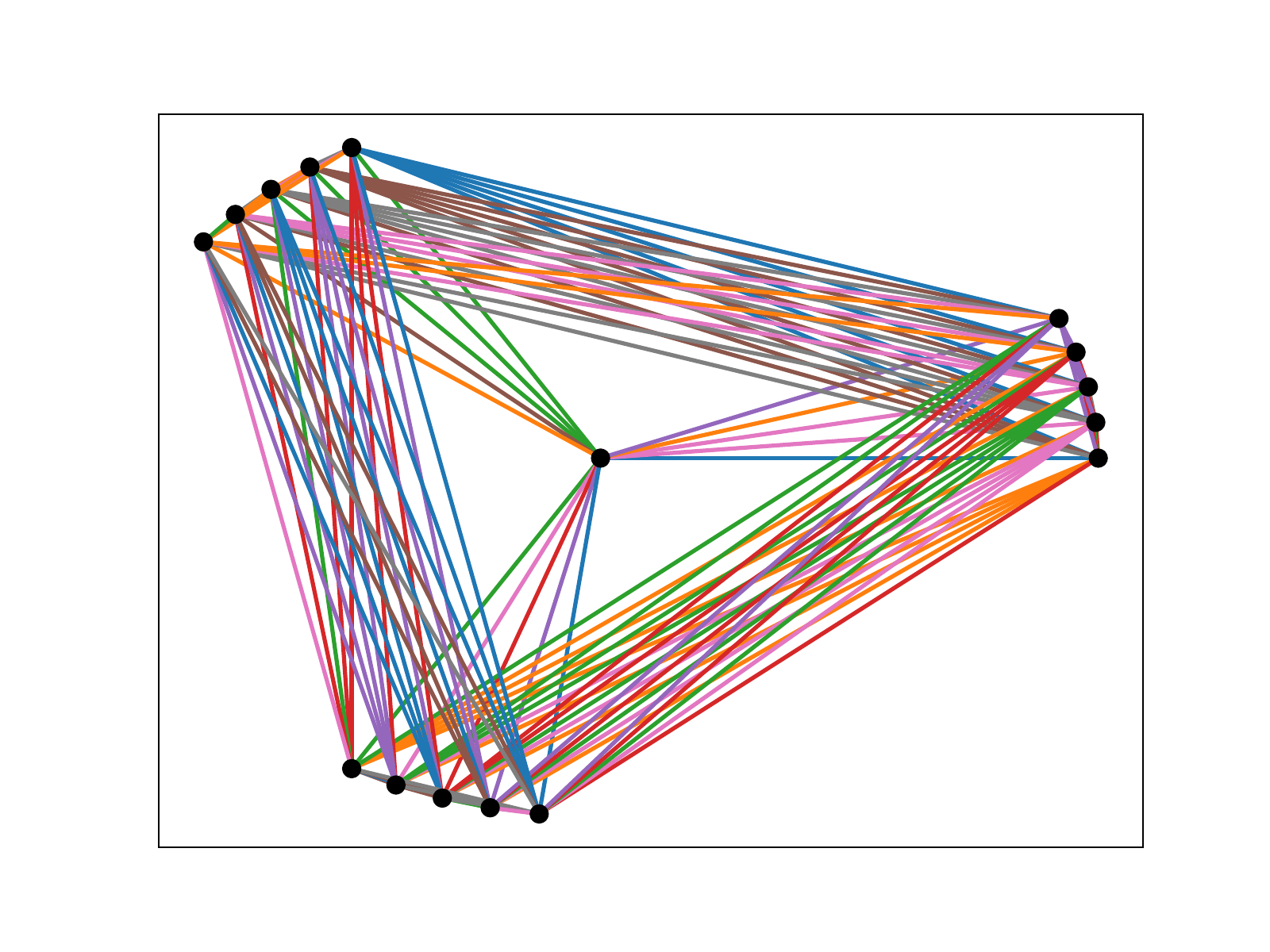}
\end{subfigure}%
\begin{subfigure}{.67\textwidth}
\centering
\includegraphics[width=0.24\textwidth,page=2]{bw_35}
\includegraphics[width=0.24\textwidth,page=3]{bw_35}
\includegraphics[width=0.24\textwidth,page=4]{bw_35}
\includegraphics[width=0.24\textwidth,page=5]{bw_35} \\
\includegraphics[width=0.24\textwidth,page=6]{bw_35}
\includegraphics[width=0.24\textwidth,page=7]{bw_35}
\includegraphics[width=0.24\textwidth,page=8]{bw_35}
\includegraphics[width=0.24\textwidth,page=9]{bw_35}
\end{subfigure}%
\caption{Partition of the bumpy wheel $BW_{3,5}$ into 8 plane subgraphs. A partition into plane spanning trees is not possible.}
\label{fig:bw_35}
\end{figure}

\section{Proof of \texorpdfstring{\Cref{thm:double_stars}}{} (partition into double stars)}\label{app:sec:double_stars}

Towards the proof of \Cref{thm:double_stars}, we need a couple of definitions and results from~\cite{schnider2016packing}. Given a partition of the edge set of a complete geometric graph into double stars, the collection of spines of the double stars forms a perfect matching, called the \emph{spine matching} (\cite{schnider2016packing},~Lemma~2).
Let $e$ be an edge between two points $p$ and $q$. The \emph{supporting line} $\ell_e$ of $e$ is the line through $p$ and $q$.
Let $e$ and $f$ be two edges and let $s$ be the intersection of their supporting lines. Recall that we say $e$ and $f$ \emph{cross}, if $s$ lies in both $e$ and $f$. If $s$ lies in $f$ but not in $e$, we say that $e$ \emph{stabs} $f$ and we call the vertex of $e$ that is closer to $s$ the \emph{stabbing vertex} of $e$. If $s$ lies neither in $e$ nor in $f$, or even at infinity, we say that $e$ and $f$ are \emph{parallel}.
A \emph{stabbing chain} are three edges, $e$, $f$ and $g$, where $e$ stabs $f$ and $f$ stabs $g$. We call $f$ the middle edge of the stabbing chain.
Note that a stabbing chain implies two interior points, so in our setting of wheel sets, there are no stabbing chains.

A \emph{cross-blocker} is a perfect matching on six points in wheel position (exactly one point in the convex hull), where the edge $e$ connecting to the interior point $v_0$ stabs both other edges $f$ and $g$, $f$ and $g$ cross, and $v_0$ is not in the convex hull of $f$ and $g$.

It is known that a spine matching can neither contain two parallel edges, nor a cross-blocker (\cite{schnider2016packing}, Theorem 9).\footnote{There is a third configuration that cannot occur, but this configuration requires a stabbing chain, so when only considering wheel sets, we may ignore this.}
On the other hand, it turns out that for wheel sets, these two configurations are the only two obstructions.
Recall the following Theorem from \cite{schnider2016packing}:

\begin{theorem}[Theorem 11 from \cite{schnider2016packing}]\label{app:thm:spine_matching}
Let $\mathcal{P}$ be a point set and let $M$ be a perfect matching on $\mathcal{P}$, such that
\begin{description}
\item[(a)] no two edges are parallel,
\item[(b)] if an edge $e$ stabs two other edges $f$ and $g$, then the respective stabbing vertices of $e$ lie inside the convex hull of $f$ and $g$, and
\item[(c)] if there is a stabbing chain, then the stabbing vertex of the middle edge lies inside the convex hull of the other two edges.
\end{description}
Then $M$ is a spine matching.
\end{theorem}

In the setting of wheel sets, case (c) cannot occur, whereas cases (a) and (b) correspond exactly to the obstructions mentioned above.
We thus get the following characterization of generalized wheel sets that allow a partition into double stars (recall that we use $GW_N$ for the underlying point set and the complete geometric graph):

\begin{corollary}
A generalized wheel $GW_N$ can be partitioned into plane spanning double stars if and only if it admits a perfect geometric matching that contains neither two parallel edges nor a cross-blocker.
\end{corollary}

Consider a generalized wheel $GW_N$ with $2n$ vertices and interior vertex $v_0$.
Let us try to construct a spine matching on $GW_N$.
To this end, we first connect $v_0$ to some other point with an edge $e$.
Note that the remaining points are now in convex position, so there is a unique matching on them without parallel edges, that is, the matching consisting of the halving edges of $GW_N\setminus\{v_0,v_1\}$.
Thus, for each choice of $e$ we get a unique possible matching, which we call a \emph{potential matching}, and this matching is a spine matching unless some edge is parallel to $e$ or there are two edges that do not have $v_0$ in their convex hull.
In the following, we investigate the conditions, under which these cases occur.

Consider a non-radial halving edge $h$ of $GW_N$. Then we call $\cl(h^-)$ a \emph{bad halfplane}.
Note that there might be no bad halfplanes, for example if $GW_N$ is a regular wheel.

\begin{lemma}
\label{lem:crosses_bad_hp}
Let $GW_N$ be a generalized wheel and assume it has a bad halfplane $B$ bounded by an edge $h$.
Assume $M$ is a spine matching on $GW_N$ which contains the edge $e=(v_0,v_1)$\footnote{We use $(x,y)$ instead of $\{x,y\}$ to denote edges in this section, trying to make it clearer that $GW_N\setminus\{x,y\}$ stands for removing two vertices from the graph.}.
Then the vertex $v_1$ lies in the bad halfplane $B$.
\end{lemma}

\begin{proof}
Assume for the sake of contradiction that $v_1$ does not lie in $B$.
Then, $GW_N\setminus\{v_0,v_1\}$ contains two more vertices in $B$ than the other (closed) side of $h$. In particular, both endpoints of $h$ (concerning the spine matching $M$) connect to vertices in $B$ through edges $f$ and $g$, since by the above arguments $f$ and $g$ need to be halving edges in $GW_N\setminus\{v_0,v_1\}$. But then $h$ separates $f$ and $g$ from $e$, and thus either one of them is parallel to $e$ or the three of them form a cross-blocker. So $M$ was not a spine matching, which is a contradiction.
\end{proof}

From this, we immediately get the following corollary:

\begin{corollary}
If $GW_N$ has three bad halfplanes whose intersection is empty, then $GW_N$ cannot be partitioned into plane spanning double stars.
\end{corollary}

\begin{proof}
By Lemma \ref{lem:crosses_bad_hp}, the vertex connected to $v_0$ has to lie in the intersection of all bad halfplanes, implying that the intersection of all bad halfplanes has to be non-empty.
\end{proof}

We will now prove that the other direction holds as well.
For this we first need some preliminary lemmas.

\begin{lemma}
\label{lem:parallels}
Let $e=(v_0,v_1)$ be a halving edge of $GW_N$.
Then the potential matching defined by $e$ has no parallel edges.
\end{lemma}

\begin{proof}
By the construction of the potential matching, there can only be pairs of parallel edges including $e$.
As any edge $f$ of the potential matching is a halving edge of $GW_N\setminus\{v_0,v_1\}$, if it was parallel to $e$, then $e$ would have more vertices on the side of $f$ as on the other side, so $e$ would not be a halving edge.
This is a contradiction.
\end{proof}

The following lemma follows from a standard rotation argument. Also remember that $GW_N$ contains $2n$ vertices.

\begin{lemma}
\label{lem:intersection}
Let $A$ be a non-empty intersection of bad halfplanes.
Then $A$ contains a vertex $v_i$ such that $(v_0,v_i)$ is a halving edge of $GW_N$.
\end{lemma}

\begin{proof}
Note that the intersection of the bad halfplanes contains a vertex of $GW_N$.
Let $u$ and $w$ be the first and the last vertex of $GW_N$ in this intersection (in clockwise order). Note that $u$ and $w$ are both incident to one of the respective halving edges each. Then, since the bad halfplanes do not contain $v_0$ by definition and hence the corresponding radial edges define strict supersets on the contained vertices, $(v_0,u)$ contains at least $n$ vertices on the right side and $(v_0,w)$ contains at least $n$ vertices on the left side (that is, the side containing the vertices in $A$, respectively). Rotating a line $\ell$ through $v_0$ from $u$ to $w$ (alongside $A$) will therefore find a radial halving edge of $GW_N$ with endpoint in $A$ (because hitting a vertex in $A$ decreases the number of vertices on the right side of $\ell$, while hitting a vertex on the opposite side of $A$ subsequently increases that value).
\end{proof}

We are now ready to prove the other direction.

\begin{lemma}
If $GW_N$ cannot be partitioned into plane spanning double stars, then $GW_N$ has three bad halfplanes whose intersection is empty.
\end{lemma}

\begin{proof}
Assume that $GW_N$ cannot be partitioned into plane double stars, that is, for every radial edge $e$ the resulting potential matching contains either an edge parallel to $e$ or two edges whose convex hull does not contain $v_0$.

Since every vertex must be incident to at least one halving edge (again using a standard rotation argument), we can now consider the potential matching $M$ defined by some radial halving edge $e=(v_0,v_1)$.
By Lemma \ref{lem:parallels}, $M$ has no parallel edges.
Since $M$ is not a spine matching by assumption and it does not contain any parallel edges, it has to contain a cross blocker. Further we can assume that the two diagonal edges $f,f'$ in this cross blocker are consecutive (in circular order), that is, the endpoints are consecutive along the convex hull. Then the unique minimal edge $h$ such that $f <_c h$ and $f' <_c h$ is a halving edge in $GW_N$. Hence, there exists a bad halfplane (which does not contain $v_1$).

We claim that the intersection of all bad halfplanes must be empty.
Indeed, if it was not empty, then by Lemma \ref{lem:intersection} the intersection would contain a point $v_i$ such that $(v_0,v_i)$ is a halving edge.
But then, by the above arguments, there is a bad halfplane which does not contain $v_i$, which is a contradiction.

As the halfplanes are convex, it now follows from Helly's theorem that if the whole family has empty intersection, then there are some three bad halfplanes whose intersection is already empty (also note that there cannot be two bad halfplanes with empty intersection).
\end{proof}

To summarize, we get:

\begin{corollary}\label{lem:bad_halfplanes}
A generalized wheel $GW_N$ cannot be partitioned into plane spanning double stars if and only if $GW_N$ has three bad halfplanes whose intersection is empty.
\end{corollary}

Finally, we are ready to prove \Cref{thm:double_stars}, which we restate here for easier readability:

\thmDoubleStars*

\begin{proof}
Let $GW_N$ be a generalized wheel with $k$ groups and $2n$ vertices. For the proof it is more convenient to consider everything from the side of the (complementary) $\frac{k+1}{2}$ consecutive groups. That is, by \Cref{lem:bad_halfplanes}, it is enough to show that $GW_N$ contains three bad halfplanes whose intersection is empty, if and only if there are three families of $\frac{k+1}{2}$ consecutive groups, each containing at least $n+1$ vertices, such that no group is in all three families.

For the one direction, assume there are three bad halfplanes whose intersection is empty. Let $h_1, h_2, h_3$ be the three respective halving edges. Next, consider for each of the three halving edges, a maximal diagonal edge $f_i$ (of distance $d_1$) with $h_i <_c f_i$ (for $i = 1,2,3$). Clearly, the closure of each $f_i^-$ contains $\frac{k+1}{2}$ consecutive groups and at least $n+1$ points. It remains to show that there is no group contained in all $f_i^-$. To this end, note first that any pair of bad halfplanes overlaps, that is, contains vertices of $GW_N$ in their intersection. Therefore, the union of the three bad halfplanes covers the entire convex hull of $GW_N$, which also holds for the union of the three $\cl(f_i^-)$. Assume for the contrary that there is a common group in the intersection of the three maximal diagonal edges, that is, $\bigcap \cl(f_i^-)$ is non-empty. Let $x$ be a point in $\bigcap \cl(f_i^-)$ and $x'$ be the antipodal point. Then $x'$ lies in $f_i^+$ for any $i$ and hence the convex hull is not fully covered, a contradiction.

For the other direction, let $F_1,F_2,F_3$ be three families of $\frac{k+1}{2}$ consecutive groups each containing at least $n+1$ vertices such that no group is in all three families. Also let $f_1,f_2,f_3$ be the (maximal) diagonal edges bounding these families. Clearly, each $\cl(f_i^-)$ contains a halving edge $e_i$. It remains to show that the intersection of the corresponding bad halfplanes is empty. Note that a non-empty intersection of three bad halfplanes must also contain a vertex of $GW_N$. Then, the corresponding group would be a common group of the three~$F_i$'s.
\end{proof}

\section{Dropping the geometric regularity}\label{app:sec:drop_geometry}

In this section, we illustrate how to drop the geometric regularity of (generalized) wheel sets. Similarly to \Cref{sec:BW} we will use the notation $e^-$ for the (open) halfplane defined by (the supporting line through) $e$ and not containing $v_0$ (the only vertex inside the convex hull). Note that we do not need an even number of vertices here, so we consider point sets on $n$ vertices now (instead of $2n$).

\begin{theorem}
Let $P$ be a set of $n$ points in the plane (in general position) with exactly one vertex $v_0$ inside the convex hull (so $n \geq 4$). Then there exists a generalized wheel $GW_N$ having the same set of crossing edge pairs as $P$.
\end{theorem}

\begin{figure}[htb]
\centering
\begin{subfigure}{.29\textwidth}
\centering
\includegraphics[width=\textwidth,page=2]{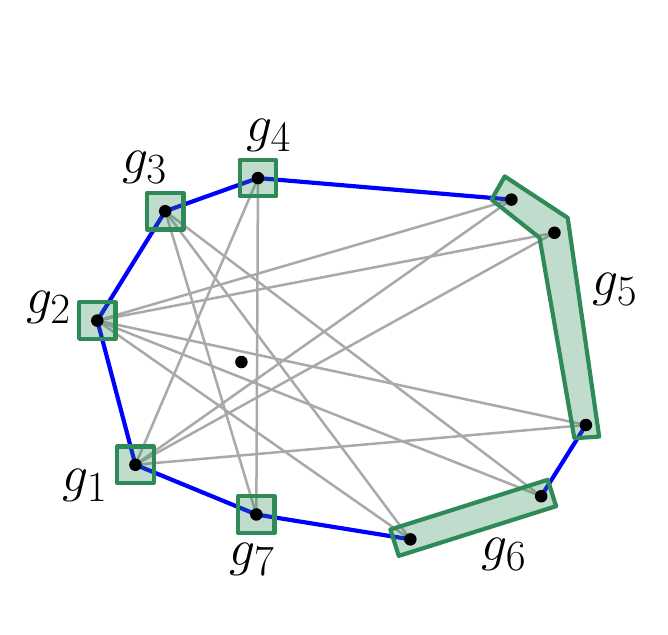}
\caption{\centering}
\label{fig:generalized_wheel_a}
\end{subfigure}\hspace{.05\textwidth}%
\begin{subfigure}{.29\textwidth}
\centering
\includegraphics[width=\textwidth,page=4]{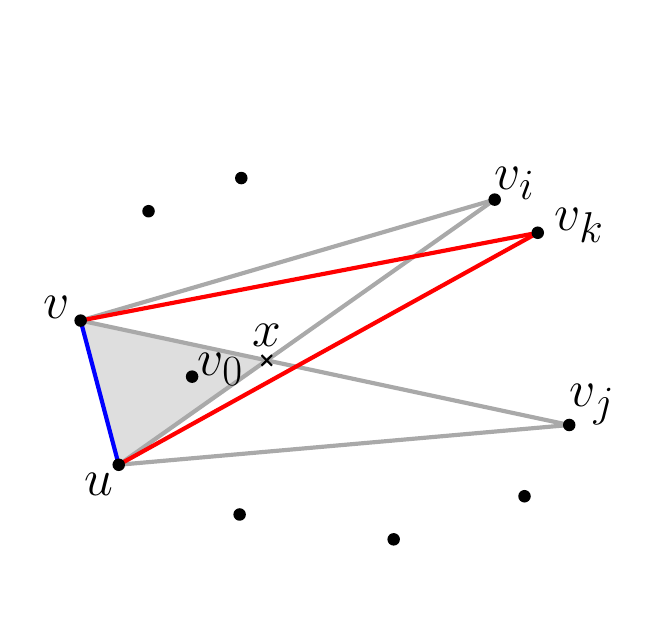}
\caption{\centering}
\label{fig:generalized_wheel_b}
\end{subfigure}\hspace{.05\textwidth}%
\begin{subfigure}{.29\textwidth}
\centering
\includegraphics[width=\textwidth,page=3]{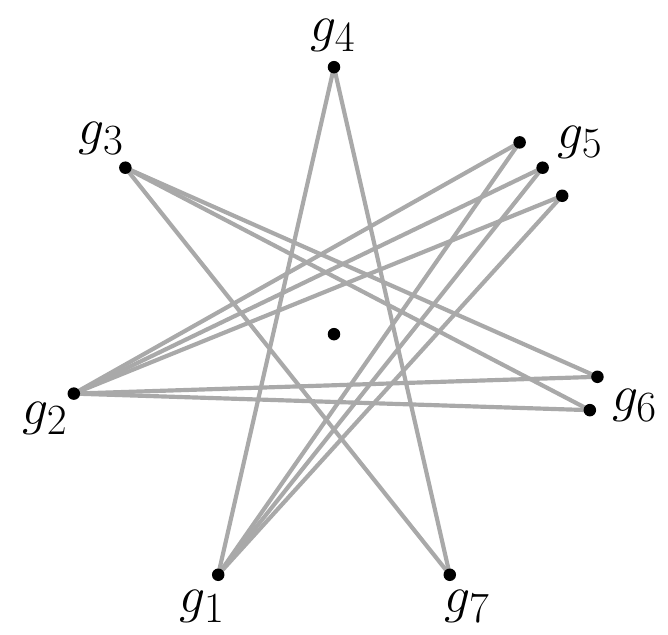}
\caption{\centering}
\label{fig:generalized_wheel_c}
\end{subfigure}
\caption{(a) Defining the groups (marked in green); opposite boundary edges drawn blue. (b) For any vertex $v_k$ between $v_i$ and $v_j$, the triangle $\Delta xuv$ has to be contained in $\Delta v_kuv$. (c) The corresponding generalized wheel.}
\label{fig:generalized_wheel}
\end{figure}

\begin{proof}
Denote the vertices on the convex hull by $v_1, \ldots, v_{n-1}$ in clockwise order (around~$v_0$) and for each vertex $v_i$ ($1 \leq i \leq n-1$) define the \emph{opposite boundary edge} (denoted by $\obe(v_i)$) as the unique boundary edge $\{v_j, v_{j+1}\}$ such that $v_0~\in~\Delta v_iv_jv_{j+1}$.
Further define the groups as sets of vertices having the same opposite boundary edge (see \Cref{fig:generalized_wheel_a}). We will see that these groups correspond precisely to the groups in the generalized wheel. Note that not all vertices can be in the same group (for example, the endpoints of $\obe(v_i)$ cannot be in the same group as $v_i$).

First, we show that each of those groups consists of consecutive vertices in~$P$ (along the convex hull). Let $v_i$, $v_j$ ($i < j$) be two vertices on the convex hull with $\obe(v_i) = \obe(v_j) (=\{u,v\})$. We claim that all vertices of~$P$ in $e^-$ of $e = \{v_i, v_j\}$ (w.l.o.g. $v_{i+1}, \ldots, v_{j-1}$) belong to the same group as $v_i$ (and $v_j$). Indeed, for any $v_k$ with $i<k<j$ we have $\Delta xuv \subseteq \Delta v_kuv$ (where $x$ is the intersection of $\{v_i,u\}$ and $\{v_j,v\}$); see \Cref{fig:generalized_wheel_b}. And since $v_0 \in \Delta xuv$, the claim follows.

Next, we show that for each opposite boundary edge $\obe(v_k) = \{u, v\}$ (of some hull vertex~$v_k$) the vertices $u$ and $v$ belong to different groups. Indeed, since $v_0~\in~\Delta v_kuv$, we get that $\obe(u)$ must lie in $\cl(e_v^-)$ of the edge $e_v = \{v, v_k\}$ (see also \Cref{fig:generalized_wheel_b}). Similarly, $\obe(v)$ must lie in $\cl(e_u^-)$ of the edge $e_u = \{u, v_k\}$. So clearly $u$ and $v$ have different opposite boundary edges.

Further, we show the following two properties by induction on the number of vertices $m = n-1$ on the convex hull:

\begin{enumerate}[(P1)]
\item\label{p1_geometric_regularity} $P$ defines an odd number of groups and
\item\label{p3_geometric_regularity} for each $v_i$, its opposite boundary edge $\obe(v_i)$ splits the remaining groups into equal parts with respect to the number of groups, that is, the line through $v_i$ and any point of (the interior of) $\obe(v_i)$ has equally many groups on each side (excluding the group containing $v_i$).
\end{enumerate}

For the base case $m = 3$ this is easy to verify (each of the $3$ vertices forms its own group). So, let $P$ be a point set with $m \geq 4$ hull vertices and consider $P'$ by removing an arbitrary hull vertex $v_i$ such that $P' = P\setminus{v_i}$ still contains exactly one point inside the convex hull (that is, $v_0 \not\in \Delta v_iv_{i+1}v_{i-1}$). By the induction hypothesis, $P'$ fulfills properties (P\ref{p1_geometric_regularity}) and (P\ref{p3_geometric_regularity}). Now, insert $v_i$ back in and let $\obe(v_i) = \{u,v\}$. Also consider the (crossing) edges $\{v_{i-1},u\}$ and $\{v_{i+1},v\}$ and call their intersection $x$; further denote their intersections with $\Delta v_iuv$ by $y$ and $z$ (see \Cref{fig:case_distinction}). Then, there are 4 different regions of $\Delta v_iuv$ that may contain $v_0$ (shaded gray in \Cref{fig:case_distinction}). We consider these cases separately:

\begin{figure}[htb]
\centering
\begin{subfigure}{.29\textwidth}
\centering
\includegraphics[width=\textwidth,page=5]{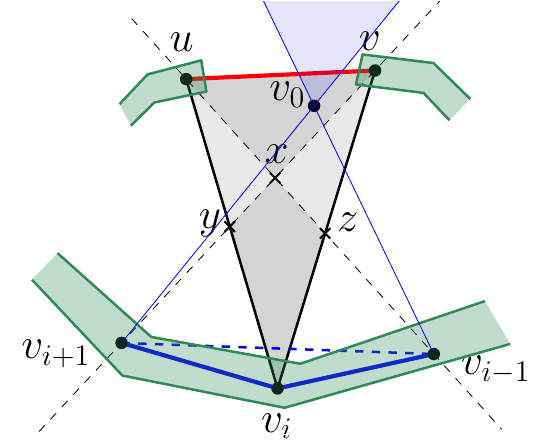}
\caption{\centering}
\label{fig:case_distinction_a}
\end{subfigure}\hspace{.05\textwidth}%
\begin{subfigure}{.29\textwidth}
\centering
\includegraphics[width=\textwidth,page=6]{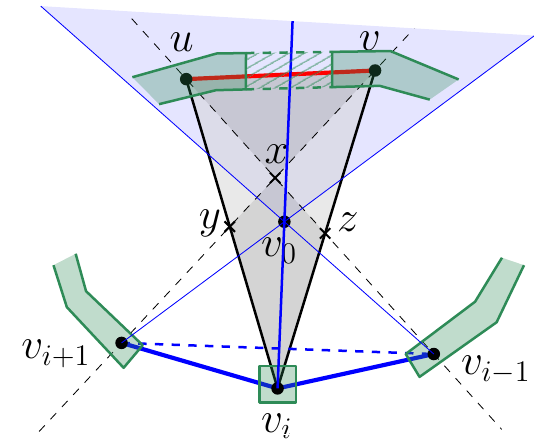}
\caption{\centering}
\label{fig:case_distinction_b}
\end{subfigure}\hspace{.05\textwidth}%
\begin{subfigure}{.29\textwidth}
\centering
\includegraphics[width=\textwidth,page=7]{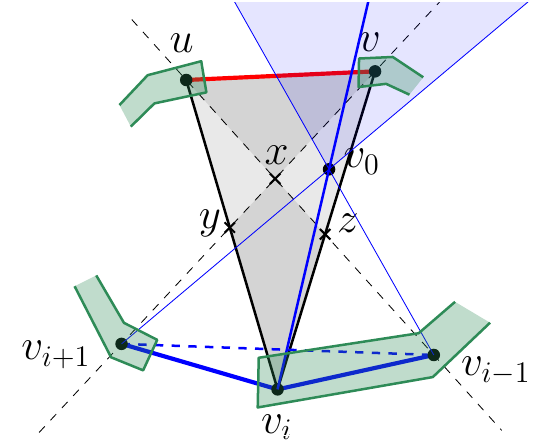}
\caption{\centering}
\label{fig:case_distinction_c}
\end{subfigure}
\caption{Illustration of the case distinction. The blue shaded region always depicts the area of vertices having opposite boundary edge $\{v_{i-1},v_{i+1}\}$. (a)~Case~1:~$v_i$ joins the already existing group; everything else remains unchanged, since the blue region does not contain any vertex. (b) Case 2: $v_i$ forms a new group of size 1 and the group containing $u$ and $v$ is split into two parts (having $\{v_{i-1},v_{i}\}$ and $\{v_{i},v_{i+1}\}$ as opposite boundary edges now; instead of $\{v_{i-1},v_{i+1}\}$). (c) Case 3: $v_i$ joins the group of $v_{i-1}$ and any vertex that had $\{v_{i-1},v_{i+1}\}$ as opposite boundary edge (the group containing $v$) has now $\{v_{i},v_{i+1}\}$ instead. Here it is crucial to note that the left cone of the blue shaded region does not contain any vertex (otherwise the corresponding group in the shaded region would be split apart by the insertion of $v_i$).}
\label{fig:case_distinction}
\end{figure}

\begin{description}
\item Case 1: $v_0 \in \Delta xuv$.

Here $v_{i-1}$ and $v_{i+1}$ belong to the same group with opposite boundary edge $\{u,v\}$. So, $v_i$ joins this already existing group. Moreover, the edge $\{v_{i-1}, v_{i+1}\}$ is replaced by the two edges $\{v_{i-1}, v_{i}\}$, $\{v_{i}, v_{i+1}\}$. However, due to the convex position of the hull vertices there cannot be any vertex having any of these as opposite boundary edge (see \Cref{fig:case_distinction_a}). In total, the number of groups remains unchanged, hence (P\ref{p1_geometric_regularity}) holds. Further (P\ref{p3_geometric_regularity}) holds by induction hypothesis (clearly for any vertex other than $v_i$, and for $v_i$ since it has the same opposite boundary edge as $v_{i-1}$ and $v_{i+1}$).

\item Case 2: $v_0 \in \lozenge v_iyxz$.

In this case, $\obe(v_{i-1})$ lies in $e_u^-$ of the edge $e_u = \{u, v_i\}$ and $\obe(v_{i+1})$ lies in $e_v^-$ of the edge $e_v = \{v, v_i\}$. Hence, $v_{i-1}$ and $v_{i+1}$ are in different groups already in $P'$ and in addition $v_i$ forms a new group of size one in $P$. Furthermore, $u$ and $v$ belong to the same group in $P'$ (with opposite boundary edge $\{v_{i-1}, v_{i+1}\}$). However, after inserting $v_i$, this group will be split into two parts --- the $u$ part with opposite boundary edge $\{v_i, v_{i-1}\}$ and the $v$ part with opposite boundary edge $\{v_i, v_{i+1}\}$ (see \Cref{fig:case_distinction_b}). All other groups stay the same. So in total, the number of groups increases by two, that is, it remains odd (confirming (P\ref{p1_geometric_regularity})).

Concerning (P\ref{p3_geometric_regularity}), each vertex in the group containing $u$ gains precisely one group on each side of its line through the new opposite boundary edge $\{v_i, v_{i-1}\}$, namely the group containing $v_i$ on the one side and the group containing $v$ on the other side. The same holds for each other vertex in $e_u^-$ (without any change to the opposite boundary edge). Similarly, the vertices in the group containing $v$ and all other vertices in $e_v^-$ get exactly one additional group on each side as well. Finally, $v_i$ fulfills (P\ref{p3_geometric_regularity}) because (P\ref{p3_geometric_regularity}) holds for $u$ (and $v$) in $P'$ and any line through $v_i$ and (the interior of) $\{u, v\}$ has the same groups on each side (in $P$) as any line through $u$ and (the interior of) $\{v_{i-1}, v_{i+1}\}$ (in $P'$), with the addition of the group containing $u$ on the one side and the group containing $v$ on the other side. Hence, also (P\ref{p3_geometric_regularity}) holds for all vertices.

\item Case 3: $v_0 \in \Delta xvz$.

Similar as before, $v_{i-1}$ and $v_{i+1}$ belong to different groups (already in $P'$). However, $v_i$ now joins the group of $v_{i-1}$ (in $P$). Moreover, the group that had $\{v_{i-1},v_{i+1}\}$ as opposite boundary edge (that is, the group containing $v$), now has the opposite boundary edge $\{v_{i},v_{i+1}\}$ and there is no vertex having $\{v_{i-1},v_{i}\}$ as opposite boundary edge (see \Cref{fig:case_distinction_c}). Therefore, the total number of groups remains unchanged (confirming (P\ref{p1_geometric_regularity})). Further, (P\ref{p3_geometric_regularity}) still holds for all vertices by induction hypothesis (also for $v_i$, since it has the same opposite boundary edge as $v_{i-1}$) because the groups remain unchanged except for $v_i$ joining the group of $v_{i-1}$ and the opposite boundary edge $\{v_{i-1},v_{i+1}\}$ getting replaced by $\{v_{i},v_{i+1}\}$.

\item Case 4: $v_0 \in \Delta xyu$.

This is analogous to Case 3 (just laterally reversed).
\end{description}

Finally, by property (P\ref{p1_geometric_regularity}) we can define the generalized wheel $GW_N$ (recall that we need an odd number of groups here) with the exact same group sizes in the same circular order as defined for $P$ (see \Cref{fig:generalized_wheel_c}). Further, by property (P\ref{p3_geometric_regularity}) we know that for each vertex $v_i$ on the convex hull, $\obe(v_i)$ is the same for $P$ and $GW_N$ (that is, for both point sets $v_0$ appears at the same spot in the rotation around $v_i$). So the two point sets have the same rotation system and therefore clearly also the same set of crossing edge pairs.
\end{proof}

\section{Additional Figures}\label{app:sec:additional_figures}

\begin{figure}[htb]
\centering
\includegraphics[width=0.49\textwidth,page=1]{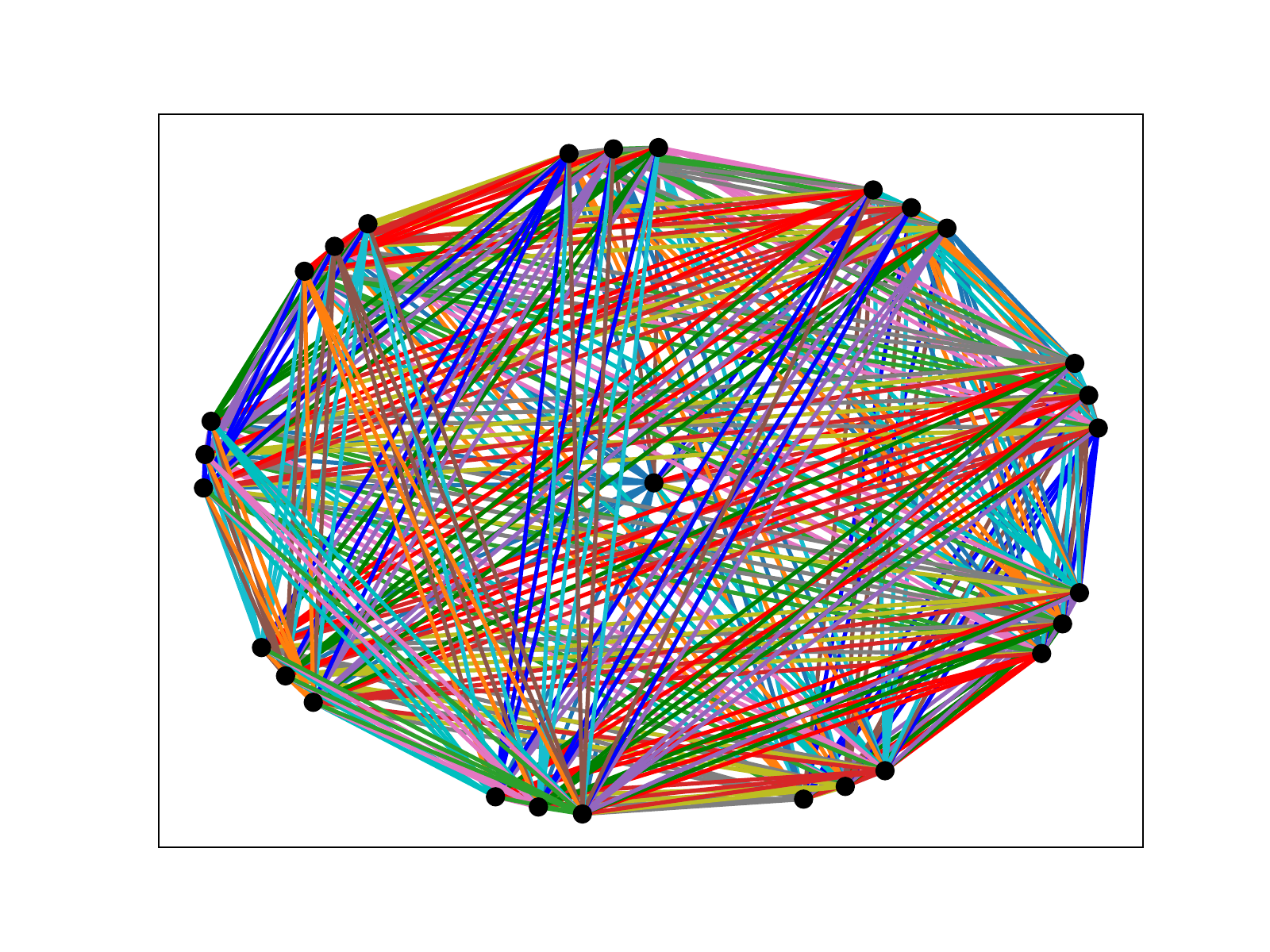}
\includegraphics[width=0.24\textwidth,page=2]{ILP_solutions/k3/three_base_cases/case1}
\includegraphics[width=0.24\textwidth,page=3]{ILP_solutions/k3/three_base_cases/case1} \\
\includegraphics[width=0.24\textwidth,page=4]{ILP_solutions/k3/three_base_cases/case1}
\includegraphics[width=0.24\textwidth,page=5]{ILP_solutions/k3/three_base_cases/case1}
\includegraphics[width=0.24\textwidth,page=6]{ILP_solutions/k3/three_base_cases/case1}
\includegraphics[width=0.24\textwidth,page=7]{ILP_solutions/k3/three_base_cases/case1} \\
\includegraphics[width=0.24\textwidth,page=8]{ILP_solutions/k3/three_base_cases/case1}
\includegraphics[width=0.24\textwidth,page=9]{ILP_solutions/k3/three_base_cases/case1}
\includegraphics[width=0.24\textwidth,page=10]{ILP_solutions/k3/three_base_cases/case1}
\includegraphics[width=0.24\textwidth,page=11]{ILP_solutions/k3/three_base_cases/case1} \\
\includegraphics[width=0.24\textwidth,page=12]{ILP_solutions/k3/three_base_cases/case1}
\includegraphics[width=0.24\textwidth,page=13]{ILP_solutions/k3/three_base_cases/case1}
\includegraphics[width=0.24\textwidth,page=14]{ILP_solutions/k3/three_base_cases/case1}
\includegraphics[width=0.24\textwidth,page=15]{ILP_solutions/k3/three_base_cases/case1}
\caption{Full partition of $BW_{9,3}$ into 14 plane spanning trees according to the construction in case 1 of \Cref{thm:bw_k3_classification} (generated by the computer assisted ILP).}
\label{fig:ilp_solutions_bw_93_case1}
\end{figure}

\begin{figure}[htb]
\centering
\includegraphics[width=0.49\textwidth,page=1]{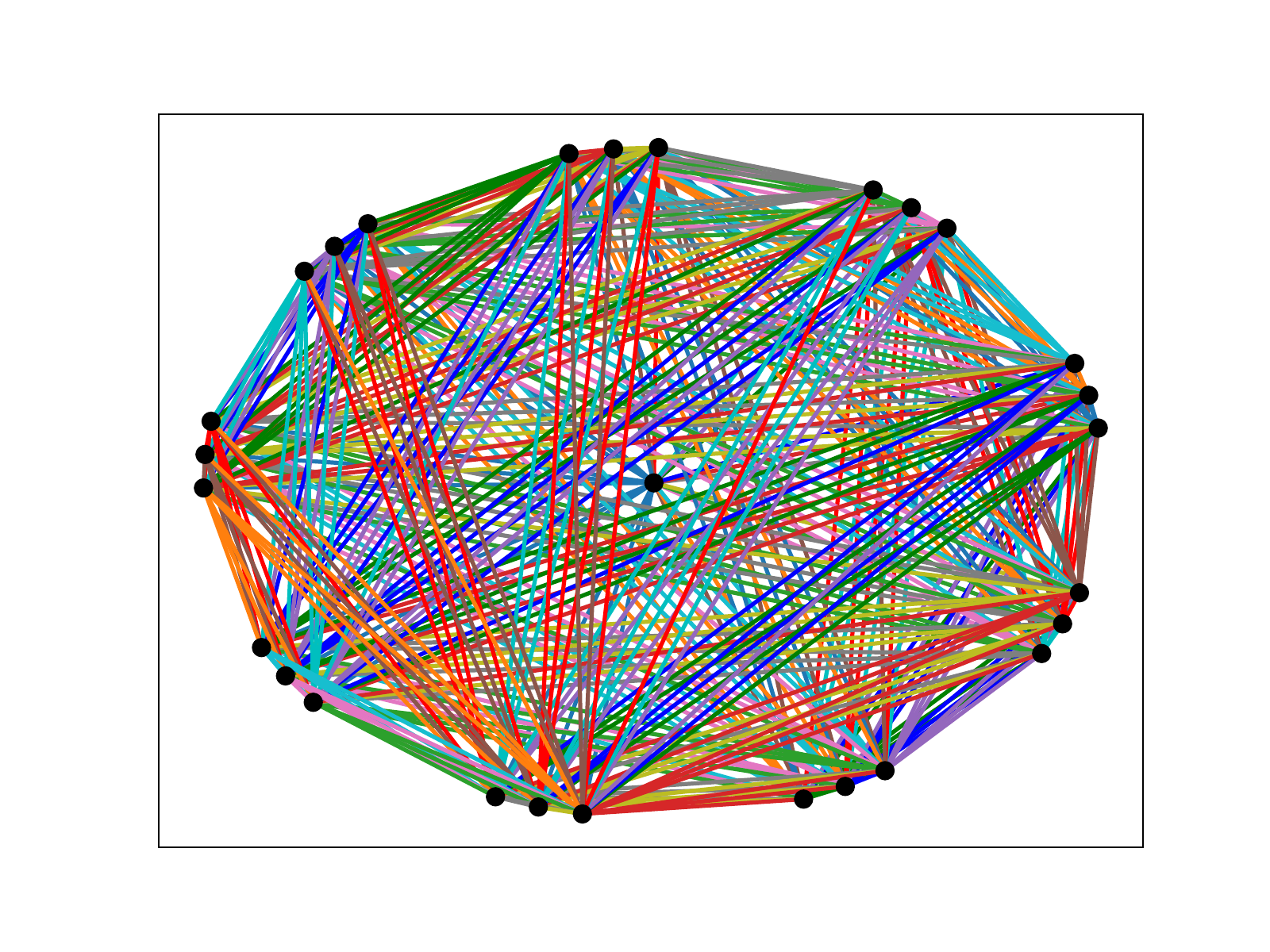}
\includegraphics[width=0.24\textwidth,page=2]{ILP_solutions/k3/three_base_cases/case2a}
\includegraphics[width=0.24\textwidth,page=3]{ILP_solutions/k3/three_base_cases/case2a} \\
\includegraphics[width=0.24\textwidth,page=4]{ILP_solutions/k3/three_base_cases/case2a}
\includegraphics[width=0.24\textwidth,page=5]{ILP_solutions/k3/three_base_cases/case2a}
\includegraphics[width=0.24\textwidth,page=6]{ILP_solutions/k3/three_base_cases/case2a}
\includegraphics[width=0.24\textwidth,page=7]{ILP_solutions/k3/three_base_cases/case2a} \\
\includegraphics[width=0.24\textwidth,page=8]{ILP_solutions/k3/three_base_cases/case2a}
\includegraphics[width=0.24\textwidth,page=9]{ILP_solutions/k3/three_base_cases/case2a}
\includegraphics[width=0.24\textwidth,page=10]{ILP_solutions/k3/three_base_cases/case2a}
\includegraphics[width=0.24\textwidth,page=11]{ILP_solutions/k3/three_base_cases/case2a} \\
\includegraphics[width=0.24\textwidth,page=12]{ILP_solutions/k3/three_base_cases/case2a}
\includegraphics[width=0.24\textwidth,page=13]{ILP_solutions/k3/three_base_cases/case2a}
\includegraphics[width=0.24\textwidth,page=14]{ILP_solutions/k3/three_base_cases/case2a}
\includegraphics[width=0.24\textwidth,page=15]{ILP_solutions/k3/three_base_cases/case2a}
\caption{Full partition of $BW_{9,3}$ into 14 plane spanning trees according to the construction in case 2a of \Cref{thm:bw_k3_classification} (generated by the computer assisted ILP).}
\label{fig:ilp_solutions_bw_93_case2a}
\end{figure}

\begin{figure}[htb]
\centering
\includegraphics[width=0.49\textwidth,page=1]{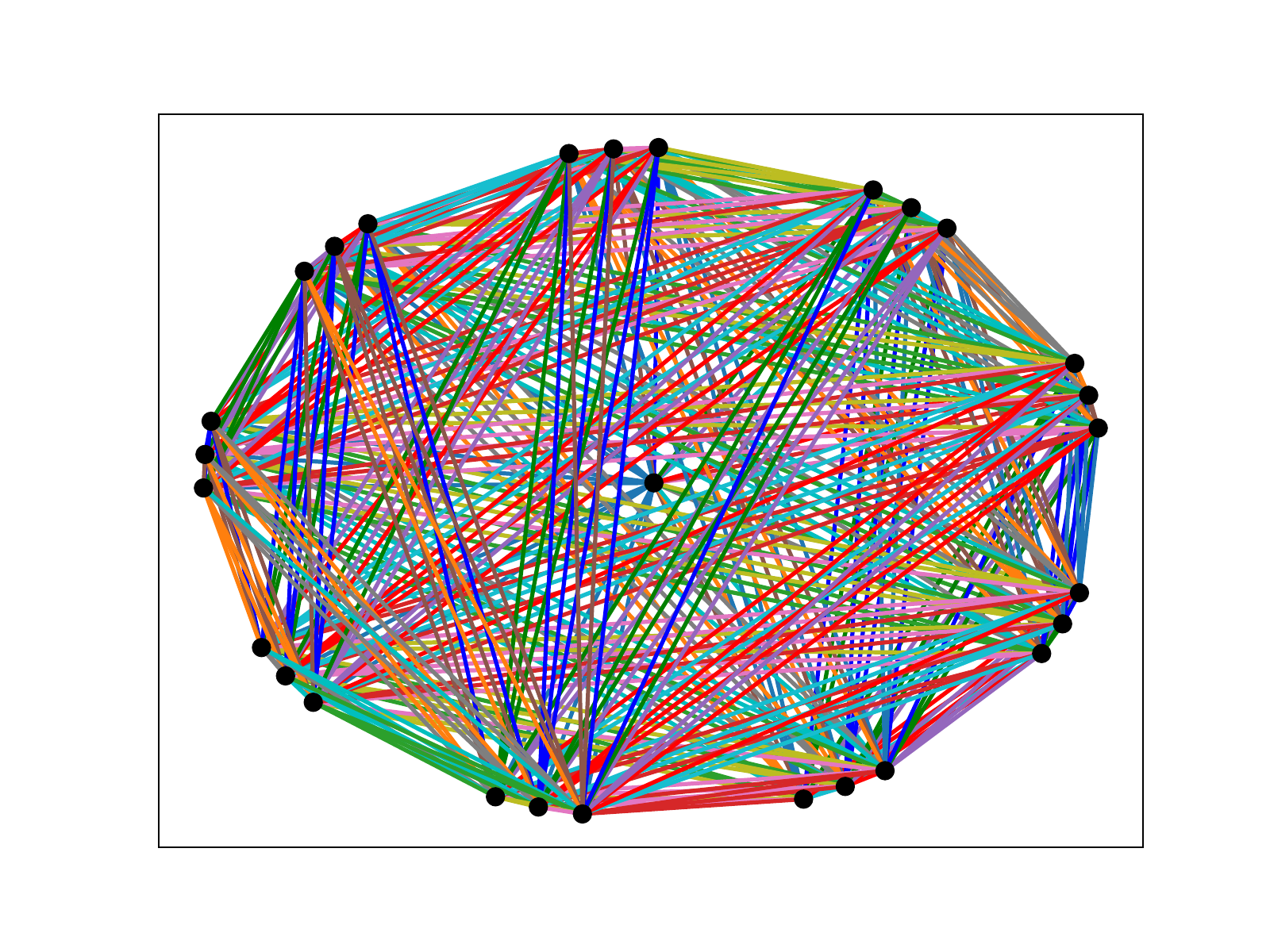}
\includegraphics[width=0.24\textwidth,page=2]{ILP_solutions/k3/three_base_cases/case2b}
\includegraphics[width=0.24\textwidth,page=3]{ILP_solutions/k3/three_base_cases/case2b} \\
\includegraphics[width=0.24\textwidth,page=4]{ILP_solutions/k3/three_base_cases/case2b}
\includegraphics[width=0.24\textwidth,page=5]{ILP_solutions/k3/three_base_cases/case2b}
\includegraphics[width=0.24\textwidth,page=6]{ILP_solutions/k3/three_base_cases/case2b}
\includegraphics[width=0.24\textwidth,page=7]{ILP_solutions/k3/three_base_cases/case2b} \\
\includegraphics[width=0.24\textwidth,page=8]{ILP_solutions/k3/three_base_cases/case2b}
\includegraphics[width=0.24\textwidth,page=9]{ILP_solutions/k3/three_base_cases/case2b}
\includegraphics[width=0.24\textwidth,page=10]{ILP_solutions/k3/three_base_cases/case2b}
\includegraphics[width=0.24\textwidth,page=11]{ILP_solutions/k3/three_base_cases/case2b} \\
\includegraphics[width=0.24\textwidth,page=12]{ILP_solutions/k3/three_base_cases/case2b}
\includegraphics[width=0.24\textwidth,page=13]{ILP_solutions/k3/three_base_cases/case2b}
\includegraphics[width=0.24\textwidth,page=14]{ILP_solutions/k3/three_base_cases/case2b}
\includegraphics[width=0.24\textwidth,page=15]{ILP_solutions/k3/three_base_cases/case2b}
\caption{Full partition of $BW_{9,3}$ into 14 plane spanning trees according to the construction in case 2b of \Cref{thm:bw_k3_classification} (generated by the computer assisted ILP).}
\label{fig:ilp_solutions_bw_93_case2b}
\end{figure}

\end{document}